\documentclass[12pt]{amsart}
\usepackage{amssymb,amsmath,amsthm,latexsym}
\usepackage{mathrsfs}
\usepackage{a4wide}
\usepackage{eucal}
\usepackage[usenames,dvipsnames]{color}
\usepackage[utf8]{inputenc}
\usepackage[T1]{fontenc}
\usepackage{dsfont}
\usepackage{enumerate}

\newtheorem{theorem}{Theorem}[section]
\newtheorem{lemma}[theorem]{Lemma}
\newtheorem{sublemma}[theorem]{Sublemma}
\newtheorem{proposition}[theorem]{Proposition}
\newtheorem{problem}[theorem]{Problem}

\newtheorem{corollary}[theorem]{Corollary}

\theoremstyle{definition}

\newtheorem{example}[theorem]{Example}

\theoremstyle{remark}
\newtheorem{remark}[theorem]{Remark}

\numberwithin{equation}{section}
\newcommand{\jns}{\textsf{JN}-sequence}

\newcommand{\JNP}{{\textsf{JNP}}}

\newcommand{\bjns}{{\textsf{BJN}-sequence}}
\newcommand{\BJNP}{{\textsf{BJNP}}}

\DeclareMathOperator{\clo}{Cl}

\newcommand{\rstr}{\restriction}

\DeclareMathOperator{\supp}{supp}

\newcommand{\hs}{\textsc{HS}}

\newcommand{\rk}{\textsf{RK}}

\newcommand{\IR}{\mathbb R}

\newcommand{\nothing}[1]{}

\newcounter{smallromansdash}

{\end{list}}

\newcounter{bigromans} 
  {\end{list}}

\begin{document}

\title[Complemented copies of $c_0$ in spaces $C_p(X\times Y)$]{On complemented copies of the space $c_0$ in spaces $C_p(X\times Y)$}
\author[J. K\k{a}kol]{J. K\k{a}kol}
\address{Faculty of Mathematics and Computer Science, A. Mickiewicz Univesity, Pozna\'n, Poland, and Institute of Mathematics, Czech Academy of Sciences, Prague, Czech Republic.}
\email{kakol@amu.edu.pl}
\author[W. Marciszewski]{W. Marciszewski}
\address{Institute of Mathematics and Computer Science, University of Warsaw, Poland Warszawa, Poland.}
\email{wmarcisz@mimuw.edu.pl}
\author[D.\ Sobota]{D. Sobota}
\address{Universit\"at Wien, Institut f\"ur Mathematik,
Kurt G\"odel Research Center, Augasse 2-6, UZA 1 --- Building 2, 1090
Wien, Austria.} \email{damian.sobota@univie.ac.at}
\author[L. Zdomskyy]{L. Zdomskyy}
\address{Universit\"at Wien, Institut f\"ur Mathematik,
Kurt G\"odel Research Center, Augasse 2-6, UZA 1 --- Building 2, 1090
Wien, Austria.} \email{lzdomsky@gmail.com}

%\urladdr{www.logic.univie.ac.at/~{}dsobota}
%\thanks{The authors were supported by the Austrian Science Fund FWF, Grant I 2374-N35.}
%\address{Institute of Mathematics, Czech Academy of Sciences, \v{Z}itn\'{a} 25, 115~67 Prague 1, Czech Republic}
%\email{tomasz.marcin.kania@gmail.com}

\begin{abstract}
Cembranos  and Freniche  proved
that for every two infinite compact Hausdorff spaces $X$ and $Y$ the
Banach space   $C(X\times Y)$ of continuous real-valued functions on
$X\times Y$ endowed with the supremum norm  contains a complemented
copy of the Banach space  $c_{0}$. We extend this theorem to the
class of $C_p$-spaces, that is, we prove that  for all  infinite
Tychonoff spaces $X$ and $Y$ the space $C_{p}(X\times Y)$ of
continuous functions on $X\times Y$ endowed with the pointwise
topology contains either a complemented copy of
$\mathbb{R}^{\omega}$ or a complemented copy of the space
$(c_{0})_{p}=\{(x_n)_{n\in\omega}\in \mathbb{R}^\omega\colon x_n\to
0\}$, both endowed with the product topology. We show that the
latter case holds always when $X\times Y$ is pseudocompact. On the
other hand, assuming the  Continuum
Hypothesis (or even a weaker set-theoretic assumption), we provide an example of  a pseudocompact space $X$ such
that  $C_{p}(X\times X)$ does
not contain a complemented copy of $(c_{0})_{p}$.

As a corollary to the first result, we show that for all infinite Tychonoff spaces $X$ and $Y$  the space  $C_{p}(X\times Y)$ is  linearly homeomorphic to the space $C_{p}(X\times Y)\times\mathbb{R}$, although, as proved earlier by Marciszewski,  there exists an infinite compact space $X$ such that $C_{p}(X)$ cannot be mapped onto  $C_{p}(X)\times\mathbb{R}$ by a continuous linear surjection. This provides a positive answer to a problem of Arkhangel'ski for spaces of the form $C_p(X\times Y)$.

Another corollary---analogous to the classical Rosenthal--Lacey theorem for Banach spaces $C(X)$ with $X$ compact and Hausdorff---asserts that for every infinite Tychonoff spaces $X$ and $Y$ the space $C_{k}(X\times Y)$ of continuous functions on $X\times Y$ endowed with the compact-open topology  admits a quotient map onto a space isomorphic to one of the following three spaces: $\mathbb{R}^\omega$, $(c_{0})_{p}$ or  $c_{0}$.
\end{abstract}

\thanks{The research for the first named
author is supported  by the GA\v{C}R project 20-22230L and RVO:
67985840. %{\color{red}The second author ?????.}
  The third and fourth authors thank the Austrian Science Fund FWF
(Grants I 2374-N35, I 3709-N35, M 2500-N35) for generous support for this
research. }

%\begin{keyword}
%Rosenthal's lemma \sep sequences of measures \sep selective ultrafilters \sep P-points \sep Q-points
%\subjclass[2010]{Primary: 28A33, 03E75. Secondary: 28E15.}
%\keywords{Grothendieck property, Grothendieck space, Boolean algebras, convergence of measures, weak topologies}
%\end{keyword}

\maketitle

\section{Introduction}

Recall that a Banach space $E$ is a \textit{Grothendieck space} or has the \textit{Grothendieck property} if every  weakly$^{*}$ convergent sequence in the dual $E^*$ of $E$ converges weakly, i.e. every sequence $(\varphi_n)_{n\in\omega}$ of continuous functionals on $E$ satisfying the condition that $\lim_{n\to\infty}\varphi_n(x)=0$ for every $x\in E$ satisfies also the condition that $\lim_{n\to\infty}\psi(\varphi_n)=0$ for every $\psi$  in $E^{**}$, the bidual space of $E$. Grothendieck \cite{grothendieck} proved  that spaces of the form $\ell_\infty(\Gamma)$ are Grothendieck spaces. Later, many other Banach spaces were recognized to be Grothendieck, e.g. von Neumann algebras (Pfitzner \cite{pfitzner}), the space $H^\infty$ of bounded analytic functions on the unit disc (Bourgain \cite{bourgain}), spaces of the form $C(K)$ for $K$ an F-space (Seever \cite{seever}), etc. On the other hand, the space $c_0$ of all sequences convergent to $0$ is not Grothendieck, since a separable Banach  space is Grothendieck if and only if it is reflexive. It follows that closed linear subspaces of Grothendieck spaces need not be  Grothendieck, although  this property is preserved by  complemented subspaces. Cembranos \cite{Ce} proved that a space $C(K)$ is Grothendieck if and only if it does not contain any complemented copy of the space $c_0$. %, see also  \cite{dales}.

%Banach spaces $E$ for which  every  weak$^{*}$ convergent sequence in the dual $E^*$ of $E$ weakly converges  were called \emph{Grothendieck spaces } or Banach spaces with the \emph{Grothendieck property}.

% Grothendieck showed also  that  every Banach space  $C(K)$
%over an extremally disconnected compact $K$ has the Grothendieck property (since then $C(K)$ is complemented in some $\ell_{\infty}(\Gamma)$). This result  has been extended  to more general classes of compact spaces, by Seever, Haydon and Freniche, see \cite{freniche}, \cite{haydon}.   We refer the reader to more "complicated"  examples of compact spaces $K$, being even Efimov spaces, for which $C(K)$ is a Grothendieck space, \cite{talagrand}, \cite{kosz}, \cite{kosz1}.

The following results due to Cembranos \cite{Ce} and Freniche \cite{fre} are the main motivation for our paper.
 \begin{theorem}[\cite{Ce}, \cite{fre}]\label{cembranos}
 Let $K$ and $L$ be infinite compact spaces and let $E$ be an infinite dimensional Banach space.
 \begin{enumerate}
 \item  $C(K,E)$  contains a complemented copy of  (the Banach space) $c_{0}$ and hence it is not a Grothendieck space.
% \item The space  $C(K)$ is Grothendieck if and only if  it does not contain complemented copy of $c_0$.
 \item The Banach space $C(K\times L)$ contains a complemented copy of $c_{0}$.
 \end{enumerate}
 \end{theorem}
The second statement of Theorem \ref{cembranos} follows from the first one
combined with the fact that $C\left( K\times L\right) $ is isomorphic to $C\left( K,C\left( L\right) \right) $.

A strengthening of Theorem \ref{cembranos}.(2) was obtained by K\k{a}kol, Sobota and Zdomskyy \cite[Corollary 11.5]{KSZ} for spaces of the form $C_p(K\times L)$ (note that it follows from the Closed Graph Theorem that a complemented copy of  $(c_0)_p$ in $C_p(K\times L)$ is actually a complemented copy of $c_0$ in $C(K\times L)$, cf. \cite[Proposition 6.1]{KSZ}). Here, by $C_{p}(X)$  and $C_{k}(X)$ we mean the space $C(X)$ of continuous real-valued functions over a Tychonoff space $X$ endowed with the pointwise and compact-open topology, respectively. $C^{*}_{p}(X)$ denotes  the vector  subspace of $C_{p}(X)$ consisting of bounded functions.
Recall that if $K$ is an infinite compact space, then one can easily construct an isometric copy of the space $c_0$ (based on disjointly supported functions) inside $C(K)$ (see \cite[Prop. 4.3.11]{AK}), but this copy (and any other) may not be complemented. The same construction applies to the space $C_p(K)$---it always contains a (closed) subspace isomorphic to $(c_{0})_{p}= \{(x_n)_{n\in\omega}\in \mathbb{R}^\omega:x_n\to 0\}$  endowed with the product topology of $\mathbb{R}^\omega$.

\begin{theorem}[{\cite[Corollary 11.5]{KSZ}}]\label{product_compact_jnp}
For every infinite compact spaces $K$ and $L$ the space $C_p(K\times L)$ contains a complemented copy of the space $(c_0)_p$.
\end{theorem}

In \cite{KSZ} many other spaces $C_p(X)$ were recognized to contain a complemented copy of $(c_0)_p$, i.a. those $C_p(X)$ where $X$ is a compact space such that the Banach space $C(X)$ is not Grothendieck. In fact, for compact spaces $K$, the existence of a complemented copy of $(c_0)_p$ in the space $C_p(K)$ appeared to be equivalent to the property that the Banach space $C(K)$ does not have the so-called $\ell_1$-Grothendieck property, a variant of the Grothendieck property defined as follows: for a given compact space $K$ we say that a Banach space $C(K)$ has the \textit{$\ell_1$-Grothendieck property} if every weakly$^*$ convergent sequence of Radon measures on $K$ with countable supports (equivalently, with finite supports) is weakly convergent; see \cite[Section 6]{KSZ} for details. Trivially, if $C(K)$ is Grothendieck, then it has the $\ell_1$-Grothendieck property; a counterexample for the reverse implication was first constructed by Plebanek (\cite{GP_unpubl}, see also \cite{Bie11}), in \cite[Section 7]{KSZ} another example and a more detailed discussion on this topic were provided.

The research in \cite{KSZ} was motivated, i.a., by the following theorem of Banakh, K\k akol and \'Sliwa \cite{BKS1}, especially by the equivalence (1)$\Leftrightarrow$(2). We refer the reader to the paper \cite{KSZ} for a detailed discussion concerning properties of sequences of measures from (1).
%
% Being motivated by Theorem \ref{cembranos}, Banakh, K\c akol, \'Sliwa \cite{BKS1} posed the following
%\begin{problem}\label{ba-ka}
%Characterize those Tychonoff spaces $X$ (or spaces $C_{p}(X)$) for which $C_{p}(X)$ contains a complemented copy of the space $(c_{0})_{p}$.
%\end{problem}
%\textcolor[rgb]{1.00,0.00,0.00}{ The following answer concerning  the above question was obtained in \cite{BKS1}.}
\begin{theorem}[{\cite[Theorem 1]{BKS1}}]\label{ba-ka1}
For a Tychonoff  space $X$ the following conditions are equivalent:
\begin{enumerate}
\item $C_{p}(X)$ satisfies the Josefson--Nissenzweig property (\JNP\,\, in short), i.e. there is a sequence $(\mu_n)_{n}$ of finitely supported signed measures  on $X$ such that $\|\mu_n\|=1$ for all $n\in \omega$, and $\mu_n(f)\to_n 0$ for each $f\in C_p(X)$.
\item $C_p(X)$ contains a complemented  subspace isomorphic to $(c_0)_{p}$;
\item $C_p(X)$ has a quotient isomorphic to $(c_0)_{p}$;
\item $C_p(X)$ admits a linear continuous map onto $(c_0)_{p}$.

\medskip

\noindent Moreover, for pseudocompact $X$ the above conditions are equivalent to:
\item $C_{p}(X)$ contains an infinite dimensional complemented  metrizable subspace.
\end{enumerate}
\end{theorem}
Consequently, if  the space $C_{p}(X)$ contains a complemented copy of $(c_0)_{p}$, the same holds for $C_{p}(X\times Y)$ for all non-empty Tychonoff spaces $Y$, since such $C_p(X\times Y)$ contains a complemented copy of $C_p(X)$. The converse implication does not hold (e.g. for $X=\beta\omega$).%, see Corollary \ref{co} below.

Condition (1) in Theorem \ref{ba-ka1} is a variant of the celebrated Josefson--Nissenzweig theorem for $C_p(X)$ spaces. Recall here that
the Josefson--Nissenzweig theorem  asserts that for each infinite-dimensional Banach space $E$ there exists a sequence $(x_n^*)_n$ in the dual space $E^*$ convergent to $0$ in the weak$^{*}$ topology of $E^*$ and such that $\|x_n^*\|=1$ for every $n\in\omega$, see e.g. \cite{Diestel} (a proof for the case $E=C(K)$ was provided in \cite[Section 3]{KSZ}, too).% \textcolor[rgb]{1.00,0.00,0.00}{ A variant of the Josefson--Nissenzweig theorem for general locally convex spaces $E$  has been recently studied in \cite{banakh-saak}.}

The main results of the present paper generalize Theorems \ref{cembranos}.(2) and \ref{product_compact_jnp} to pseudocompact spaces in the following way.

\begin{theorem}\label{solution}
Let $X$ and $Y$ be infinite Tychonoff  spaces. Then:
\begin{enumerate}
\item If the space $X\times Y$ is pseudocompact, then   $C_{p}(X\times Y)$ contains a complemented copy of $(c_{0})_{p}$.
\item If  the space  $X\times Y$ is not pseudocompact, then  $C_{p}(X\times Y)$ contains a complemented copy of $\mathbb{R}^{\omega}$.
\end{enumerate}
\end{theorem}
Consequently, if for infinite Tychonoff spaces $X$ and $Y$ the product $X\times Y$ is  pseudocompact, then $C_{p}(X\times Y)$ contains a complemented copy of $(c_{0})_{p}$, so by the Closed Graph Theorem % \cite[Theorem 4.1.10]{bonet}
the Banach space  $(C(X\times Y),\|.\|_{\infty})$ contains a
complemented copy of the Banach space $c_{0}$. This provides a
stronger version of Cembranos--Freniche theorem. However, if the
product $X\times Y$ is not pseudocompact, then consistently
$C_p(X\times Y)$ may fail to have the complemented copy of
$(c_0)_p$, because, as our next theorem shows, this  may happen even for the
squares. We refer the reader to the paragraph
before Theorem~\ref{main_ex} for the exact formulation of the
set-theoretic assumption we use in the proof of the next
theorem (and which is called $(\dagger)$ by us). Let us only mention here that this assumption $(\dagger)$ is satisfied if the Continuum Hypothesis or Martin's axiom holds.

%\textcolor[rgb]{1.00,0.00,0.00}{It is well known that  the product of a compact space with a pseudocompact space is pseudocompact, and there exists a pseudocompact space $X$ whose own square fails this property (see \cite{En}?????????). Hence,   there exists a pseudocompact  space $X$ with an infinite compact subset $K$  such that $X\times X$ is not pseudocompact. But then the restriction map $C_{p}(X\times X)\rightarrow C_{p}(K\times K)$ is a quotient surjection and  $C_{p}(K\times K)$ contains a complemented copy of  $(c_{0})_{p}$ (by Theorem \ref{solution}). Consequently,   $C_{p}(X\times X)$ does contain a complemented copy of $(c_{0})_{p}$ (by virtue of Theorem \ref{ba-ka1})  although $X\times X$ is not pseudocompact.} We have however the following
\begin{theorem}\label{lastest}
It is consistent that there exists an infinite
pseudocompact space $X$ % with no infinite compact subsets and
such that  the spaces $C_p(X\times X)$ and  $C_p^*(X\times X)$
do not  contain a complemented copy of $(c_0)_{p}$.
\end{theorem}
%Theorem \ref{lastest} will be deduced from a number of additional lemmas  and  theorems  of own interest. Among the others, a concept of a "bounded" version of the \JNP\,\, will be introduced.
%The theorem remains true if we assume Martin's axiom instead of the Continuum Hypothesis.
The proofs of Theorems \ref{solution} and \ref{lastest} are provided in Sections \ref{section_theorem_solution} and \ref{section_x_x_no_bjns}, respectively. Unfortunately, we do not know whether the negation of the conclusion of Theorem \ref{lastest} may consistently hold. We are also not aware of any model of ZFC where $(\dagger)$, our set-theoretic assumption used to prove Theorem 1.5, fails.

\begin{problem}\label{pro}
Is it consistent that for any infinite  pseudocompact space $X$ the space $C_p(X\times X)$ (respectively,  $C^{*}_p(X\times X)$)  contains a complemented copy of $(c_{0})_p$ ?
\end{problem}

Recall that every countably compact space is pseudocompact, thus the following question seems natural in the context of Theorems \ref{product_compact_jnp} and \ref{lastest}.

\begin{problem}\label{pro_cc}
Is it consistent that there exists an infinite countably compact space $X$ such that the space $C_p(X\times X)$ (respectively,  $C^{*}_p(X\times X)$) does not contain a complemented copy of $(c_0)_p$?
\end{problem}

The methods applied to prove Theorem \ref{lastest} cannot be used to answer Problem \ref{pro_cc} in the affirmative since they produce spaces which are very far from being countably compact. However, recall that there exist countably compact spaces whose squares are not pseudocompact, see e.g. \cite[Example 3.10.9]{En}.

Theorem \ref{solution} has two important applications concerning the following remarkable problem  in $C_p$-theory  posed by Arkhangel'ski, see \cite{arkh2,arkh4}, and the famous Rosenthal--Lacey theorem.

\begin{problem}[Arkhangel'ski] \label{Ar}
Is it true that for every infinite (compact) space $K$ the space  $C_{p}(K)$ is linearly homeomorphic to  $C_{p}(K)\times\mathbb{R}$?
\end{problem}

%Recall that for every Tychonoff space $X$ and any $x\in X$ the space $C_{p}(X)$ is linearly homeomorphic to the product $\{f\in C(X): f(x)=0\}\times\mathbb{R}$ {\color{red}for a fixed point $x\in X$???}, see \cite[S.182]{Tka1}. %However, there are many examples of normed spaces $E$ are known such that $E$ is not linearly homeomorphic to $E\times\mathbb{R}$, see for example \cite{BE}, \cite{BEP}. In \cite{gowers} and \cite{gowers1}  Gowers and Maurey constructed examples of Banach spaces $E$ such that $E$ is not linearly homeomorphic to $E\times\mathbb{R}$. In \cite{Mill} van Mill constructed a normed space $E$ which is even not homeomorphic to $E\times\mathbb{R}$.
For a wide class of spaces the answer to Problem \ref{Ar}  is affirmative, e.g. if the space $X$ contains a non-trivial convergent sequence, or $X$ is not pseudocompact (see \cite[Section 4]{arkh4}), yet, in general, the answer is negative  even for compact spaces $X$, see  Marciszewski \cite{marci}. % Later on,  Koszmider in \cite{kosz} constructed  Banach spaces $C(X)$  not linearly homeomorphic to $C(X)\times\mathbb{R}$. Hence (by using the classic closed graph theorem) $C_{p}(X)$ is not linearly homeomorphic to $C_{p}(X)\times\mathbb{R}$.
It appears however that for finite products of Tychonoff spaces the answer is still positive---the following corollary follows immediately from Theorem \ref{solution}.
%\textcolor[rgb]{1.00,0.00,0.00}{ Theorem \ref{solution} applies to get the following general}
\begin{corollary}\label{solution1}
Let $X$ and $Y$ be infinite Tychonoff spaces. Then  $C_{p}(X\times Y)$ is linearly homeomorphic to the product   $C_{p}(X\times Y)\times\mathbb{R}$.
\end{corollary}

The famous  Rosenthal--Lacey theorem \cite{Ro}, \cite{La}, see also \cite[Corollary 1]{hagler},  asserts that  for each infinite compact  space $K$ the  Banach space $C(K)$ admits a quotient map onto  $c_0$ or $\ell_{2}$; we refer the reader to a survey paper \cite[Theorem 18]{fe-ka-sliwa} for a detailed discussion on the theorem. The case of $C_p$-spaces remains however open, namely, it is still unknown whether  for every infinite compact space $K$ the space $C_{p}(K)$ admits a quotient map onto an  infinite-dimensional metrizable space, see \cite{kakol-sliwa}. Nevertheless, Theorem \ref{solution} yields the following corollary.

\begin{corollary}\label{separable}  Let $X$ and $Y$ be infinite Tychonoff spaces. Then
\begin{enumerate}
\item   $C_{p}(X\times Y)$ admits a quotient map onto  $\mathbb{R}^\omega$ or  $(c_{0})_{p}$.
\item   $C_{k}(X\times Y)$ admits a quotient map onto a space isomorphic to one of the following spaces: $\mathbb{R}^\omega$, $(c_{0})_{p}$ or $c_{0}$.
\end{enumerate}
\end{corollary}

The proofs of Theorem \ref{solution} and its consequences, e.g. Corollary \ref{separable}, are provided in Section \ref{section_theorem_solution}.

% Theorem \ref{cembranos}.(1) has been extended by Doma\'nski and Drewnowski \cite{drew} to \emph{Fr\'echet locally convex spaces} $E$, i.e. metrizable and complete locally convex spaces.
% \begin{theorem}[\cite{drew}]\label{doma}
% Let $X$ a Tychonoff space containing an infinite compact subset,  and let $E$ be an infinite dimensional Fr\'echet locally convex space which is not a Montel space \textcolor[rgb]{1.00,0.00,0.00}{ i.e.  contains a closed bounded non-compact set}. Then the space $C_{k}(X,E)$ contains a complemented copy of $c_{0}$.
% \end{theorem}
%
%Bearing in mind Theorem \ref{cembranos} and Theorem \ref{doma} one can formulate the corresponding
%\begin{problem}\label{finest}
% Let $X$ be an infinite compact space. For which infinite dimensional locally
%convex spaces $E$ the space $C_p(X, E_w)$ or $C_p(X, E)$ contains a complemented copy of $(c_{0})_{p}$,
%where $E_w$ means the space $E$ endowed with its weak topology?
%\end{problem}
%In particular one can ask:
%\begin{problem}\label{problem}
%Let $X$ be an infinite compact space and $E$ an infinite dimensional Fr\'echet locally convex space which is not a Montel space. Does $C_{p}(X,E)$ contain a complemented copy of $(c_{0})_{p}$?
%\end{problem}
%\textcolor[rgb]{1.00,0.00,0.00}{An affirmative answer to Problem \ref{problem}}  would provide an extension of Theorem \ref{doma}.
\medskip

\noindent\textbf{Notation and terminology.}
Our notation and terminology are standard, i.e., we follow monographs of Tkachuk \cite{Tka1} (function spaces), Engelking \cite{En} (general topology), Halmos \cite{Hal74} (measure theory) and Jech \cite{Jec02} (set theory). In particular, we assume that \textbf{all topological spaces we consider are Tychonoff}, that is, completely regular and Hausdorff. 

The cardinality of a set $X$ is denoted by $|X|$. By $\omega$ we denote the first infinite cardinal number, i.e., the cardinality of the space of natural numbers $\mathbb{N}$. Usually we identify $\omega$ with $\mathbb{N}$, so $\omega$ is an infinite countable discrete topological space, thus, e.g., such notions like the \v{C}ech--Stone compactifiactions $\beta\omega$ of $\omega$ have sense. We denote the remainder $\beta\omega\setminus\omega$ of this compactification by $\omega^*$. The continuum, i.e. the cardinality of the real line $\mathbb{R}$, is denoted both by $\mathfrak{c}$ and $2^\omega$. If $X$ is a set and $\kappa$ is a (finite or infinite) cardinal number, then by $[X]^\kappa$ we denote the family of all subsets of $X$ of cardinality $\kappa$; in particular, $[X]^{\omega}$ denotes the family of all infinite countable subsets of $X$. We put $[X]^{<\omega}=\bigcup_{n\in\omega}[X]^n$, so $[X]^{<\omega}$ is the family of all finite subsets of $X$. Finally, $\omega^\omega$ denotes the family of all functions from $\omega$ into $\omega$.

All other necessary and possibly non-standard notions will be defined in relevant places of the text.

\section{Proof of Theorem \ref{solution} and its consequences\label{section_theorem_solution}}

For a space $X$ and a point $x\in X$ let $\delta_x: C_p(X)\to\mathbb{R},\,\,\, \delta_x:f\mapsto f(x),$ be the Dirac measure concentrated at $x$. The linear hull $L_p(X)$ of the set $\{\delta_x:x\in X\}$ in $\mathbb{R}^{C_p(X)}$ can be identified with the dual space of $C_p(X)$ (\cite[Exercises 196--197]{Tka1}). Elements of the space $L_p(X)$ will be called {\em finitely supported signed measures} (or simply {\em signed measures}) on $X$.

Each $\mu\in L_p(X)$ can be uniquely written as a linear combination of Dirac measures $$\mu=\sum_{x\in F}\alpha_x\delta_x$$ for some finite set $F\subset X$ and some non-zero real numbers $\alpha_x$. The set $F$ is called the {\em support} of the signed measure $\mu$ and is denoted by $\supp(\mu)$. The measure $\sum_{x\in F}|\alpha_x|\delta_x$ will be denoted by $|\mu|$ and the real number $$\|\mu\|=\sum_{x\in F}|\alpha_x|$$ coincides with the {\em norm} of $\mu$ in the dual Banach space $C(\beta X)^*$.

A sequence $(\mu_n)_n$ of finitely supported signed measures  on $X$ such that $\|\mu_n\|=1$ for all $n\in \omega$, and $\lim\limits_{n}\mu_n(f) = 0$ for each $f\in C_p(X)$ is called {\em a Josefson--Nissenzweig sequence} or, in short, {\em a \jns} on $X$. A \jns\ $(\mu_n)_n$ is supported on a subset $A$ of a space $X$ if the supports of all measures $\mu_n$ are contained in $A$.

We say that $C_{p}(X)$ {\em has the Josefson--Nissenzweig property} or, in short, \JNP\ if $X$ admits a \jns.

The following proposition was proved in \cite{KSZ}. We provide a brief sketch of the proof for the sake of completeness. % Since this paper is not published yet, for the sake of completeness, we include a sketch of the proof of it.
\begin{proposition}[{\cite[Proposition 11.2]{KSZ}}]\label{beta^2_has_JNP}
The product $\beta\omega\times\beta\omega$ has a \jns\ $(\mu_{n})_{n\in\omega}$ supported on $\omega\times\omega$. Moreover, the supports of $\mu_n$ have pairwise disjoint projections onto each axis.
\end{proposition}
\begin{proof}
For every $n\in\mathbb{N}$ put $\Omega_n=\{-1,1\}^n$ and $\Sigma_n=n\times\{n\}$. To simplify the notation, we will usually write $i\in\Sigma_n$ instead of $(i,n)\in\Sigma_n$.

Define  $\Omega=\bigcup_{n\in\omega}\Omega_n$ and $\Sigma=\bigcup_{n\in\omega}\Sigma_n$, and endow these two sets with the discrete topology. This enables us to look at the the product space $\Omega\times\Sigma$ as a countable union of pairwise disjoint discrete rectangles $\Omega_k\times\Sigma_m$ of size $m2^k$---the rectangles $\Omega_n\times\Sigma_n$, lying along the diagonal, will bear a special meaning, namely, they will be the supports of measures from a \jns\,\,$(\mu_{n})_{n\in\omega}$ on the space $\beta\Omega\times\beta\Sigma$ defined as follows:
$$\mu_n=\sum_{\substack{s\in\Omega_n,i\in\Sigma_n}}\frac{s(i)}{n2^n}\delta_{(s,i)},\,\,\,n\in\omega.$$ It turns out that the  sequence $(\mu_{n})_{n\in\omega}$  defined above is a \jns,  see  \cite[Section 11]{KSZ} for details. Note  that $\omega$ is homeomorphic to both $\Omega$ and $\Sigma$, so $\beta\omega$, $\beta\Omega$ and $\beta\Sigma$ are mutually homeomorphic.
Consequently, $\beta\omega\times\beta\omega$ has the \jns\ with the required properties.
\end{proof}

In \cite[Theorem 11.3]{KSZ} this result was used to prove that, for any infinite compact spaces $K,L$, the space  $C_p(K\times L)$ has the \JNP. By a  modification of this argument we can obtain the following strengthening:

\begin{theorem}\label{pseudocompact_product_has_JNP}
For any infinite spaces $X,Y$, if the product $X\times Y$ is pseudocompact, then it admits a \jns\ $(\mu_{n})_{n\in\omega}$. Moreover, we can require that $(\mu_{n})_{n\in\omega}$ is supported on the product $D\times E$, where $D\subset X$ and $E\subset Y$ are countable discrete, and the supports of $\mu_n$ have pairwise disjoint projections onto each axis.
\end{theorem}

\begin{proof}
Let $D$ and $E$ be  countable discrete subsets of $X$ and $Y$, respectively. Let $\varphi: \omega\to D$ and $\psi: \omega\to E$ be bijections. By the Stone Extension Property of $\beta\omega$, there are continuous maps $\Phi: \beta\omega\to \beta X$ and $\Psi: \beta\omega\to \beta Y$ such that  $\Phi\rstr  \omega=\varphi$ and $\Psi\rstr  \omega=\psi$.

We denote by $\Theta$ the product map $$\Phi\times\Psi: \beta\omega\times\beta\omega \to \beta X\times \beta Y.$$ Clearly $\Theta$ maps $\omega\times \omega$ injectively into $X\times Y$.

Let $(\mu_n)_{n\in\omega}$ be a \jns\ of measures on $(\beta\omega)^2$ supported on $\omega^2$, given by Proposition \ref{beta^2_has_JNP}. For each $n\in\omega$, we consider the image of $\mu_n$ under $\Theta$, i.e., the measure $\nu_n$ on $X\times Y$ defined as follows:
\[\nu_n=\sum_{z\in\supp(\mu_n)}\mu_n\big(\{z\}\big)\cdot\delta_{\Theta(z)},\]
it follows that $\|\nu_n\|=1$ and $\supp(\nu_n)$ is finite. We will show that the sequence $(\nu_n)$ converges to $0$ on every $f\in C(X\times Y)$ which will demonstrate that $X\times Y$ admits a \jns.

Recall that by Glicksberg's theorem \cite[Exercise 3.12.21(c)]{En}, pseudocompactness of $X\times Y$ implies that $\beta X\times \beta Y$ is the \v Cech--Stone  compactification of $X\times Y$, i.e., every continuous function on $X\times Y$ is continuously extendable over  $\beta X\times \beta Y$.

Fix $f\in C(X\times Y)$ and let $F$ be its continuous extension over  $\beta X\times \beta Y$. Then the composition $F\circ\Theta$ is a continuous function on $\beta\omega\times\beta\omega$ and $\nu_n(f) = \mu_n(F\circ \Theta)$ for each $n\in\omega$. Therefore $\lim_{n} \nu_n(f) = 0$.

The additional properties of the supports of the measures $\nu_n$ follow easily from the definition of $\nu_n$ and the corresponding properties of supports of the measures $\mu_n$.
\end{proof}

We are in the position to prove the main result of this paper, Theorem \ref{solution}.

\begin{proof}[{\normalfont\textbf{Proof of Theorem \ref{solution}}}]
If the space $X\times Y$ is pseudocompact, then by the above theorem and Theorem \ref{ba-ka1} the space $C_p(X\times Y)$ has a complemented copy of $(c_0)_{p}$.

On the other hand, it is well known that if a space $X$ is not pseudocompact, then $C_p(X)$ has a complemented copy of $\mathbb{R}^\omega$, cf. \cite[Section 4]{arkh4} or \cite[Theorem 14]{ka-sa}. This completes the proof of part (2) of Theorem \ref{solution}.
\end{proof}
\medskip

Observe that if a subspace $Y$ of a space $X$ has a \jns, then this sequence is also a \jns\ on $X$. By Theorem \ref{ba-ka1}, it follows that if $C_{p}(Y)$ contains a complemented copy of $(c_{0})_{p}$, then $C_p(X)$ also contains such a copy. Therefore Theorem \ref{solution}.(1) immediately gives the following

\begin{corollary}\label{cor_subspaces}
If a space $Z$ contains a topological copy of a pseudocompact product $S\times T$ of infinite spaces $S$ and $T$, then  $C_{p}(Z)$ contains a complemented copy of $(c_{0})_{p}$. In particular, if spaces $X$ and $Y$ contain infinite compact subsets, then $C_{p}(X\times Y)$ contains a complemented copy of $(c_{0})_{p}$.
\end{corollary}

Note that there exist infinite spaces $X$ without infinite compact subsets such that the square $X\times X$ is pseudocompact, see Example \ref{second}.
\medskip

Recall that a subset $B$ of a locally convex space $E$ is \emph{bounded}  if for every neighbourhood of zero $U$ in $E$ there exists a scalar $\lambda >0$ such that $\lambda B\subset U.$

Note the following fact connected with the next Corollary \ref{do-ka}.
\begin{lemma}\label{pseudo}
For a space $X$ the following assertions are equivalent:
\begin{enumerate}
\item The space  $C_{k}(X)$ is covered by a sequence of bounded sets.
\item The space $C_{p}(X)$ is covered by a sequence of bounded sets.
\item   $X$  is pseudocompact.
\end{enumerate}
\end{lemma}
\begin{proof}
(1) $\Rightarrow$ (2) is clear, since the compact-open topology is stronger than the pointwise one of $C(X)$.

(2) $\Rightarrow$ (3): Assume $C_{p}(X)$ is  covered by a sequence of bounded sets  but  $X$ is not psudocompact. Then  $C_{p}(X)$ contains a complemented copy of $\mathbb{R}^{\omega}$. On the other hand,  $\mathbb{R}^{\omega}$ cannot be covered by a sequence $(S_{n})_{n\in\omega}$ of bounded sets. Indeed, we may assume that each $S_{n}$ is absolutely convex and closed and $S_{n}\subset S_{n+1}$ for all $n\in\omega$. By the Baire theorem some $S_{n}$ is a neighbourhood of zero in  $\mathbb{R}^{\omega}$. Consequently $\mathbb{R}^{\omega}$ must be a normed space by a theorem of Day, see \cite[Proposition 6.9.4]{Jarchow}, a contradiction (since  $\mathbb{R}^{\omega}$ is not normable).

(3) $\Rightarrow$ (1): If $X$ is pseudocompact and $$S=\{f\in C(X): |f(x)|\leq 1, x\in X\},$$ then the sequence of bounded sets $(nS)_{n}$ in the compact-open topology covers $C_{k}(X)$.
\end{proof}

\begin{remark}\label{remark_pseudo}
Since the image of a bounded subset of a topological vector space under a continuous linear operator is bounded, from the above lemma we can easily deduce that if $T: C_k(X)\to C_k(Y)$ ($T: C_p(X)\to C_p(Y)$) is a linear continuous surjection and the space $X$ is pseudocompact, then $Y$ is also pseudocompact. This fact for the pointwise topology is well known, cf.\ \cite[Proposition 6.8.6]{vM2}.
\end{remark}

Since the pseudocompactness is not  transferred   even to  finite products (see \cite{En}),  there exist spaces $X$ and $Y$ such that both  $C_{k}(X)$ and $C_{k}(Y)$ are covered by a sequence of bounded sets but $C_{k}(X\times Y)$ lacks this property, by Lemma \ref{pseudo}.
\medskip

A sequence $(S_{n})_{n\in\omega}$ of bounded sets in a locally convex space $E$ is \emph{fundamental} if every bounded set in $E$ is contained in some set $S_{n}$. 
%Every $(DF)$  (in particular,  every normed) space $C_{k}(X)$ admits a fundamental sequence of bounded sets.
By Warner \cite{warner}, the space $C_{k}(X)$ admits a fundamental sequence of bounded sets if and only if the following condition  holds:

\begin{itemize}
\item[($*$)] \emph{Given any sequence $(G_{n})_{n\in\omega}$ of pairwise disjoint  non-empty open subsets of $X$ there is a compact set $K\subset X$ such that $\{n\in\omega:K\cap G_{n}\neq\emptyset\}$ is infinite.}
\end{itemize}

This characterization easily implies that if the space $C_{k}(X)$ has a fundamental sequence of bounded sets, and the space $X$ is infinite, then $X$ contains an infinite compact subspace. Therefore, by Corollary \ref{cor_subspaces} we obtain

\begin{corollary}\label{do-ka}
Let $X$ and $Y$ be two  infinite spaces. If  $C_{k}(X)$ and  $C_{k}(Y)$ admit fundamental sequences of bounded sets, then $C_{p}(X\times Y)$ contains  a complemented copy of   $(c_{0})_{p}$.
\end{corollary}

\begin{remark}
Note that if  $X$ is a pseudocompact space without infinite compact subsets (cf.\ Section 3), then by Warner's characterization, $C_{k}(X)$ does not have a fundamental sequence of bounded sets, although it is covered by a sequence of bounded sets, by Lemma \ref{pseudo}.

\end{remark}
Theorem  \ref{solution} applies also to get the following
\begin{corollary}\label{maps}
Let $X$ be an infinite pseudocompact  space such that $C(\beta X)$ is a Grothendieck space. Then for no infinite spaces $Y$ and $Z$ does  exist a continuous linear surjection from $C_{p}(X)$ onto the space  $C_{p}(Y\times Z)$.
\end{corollary}
\begin{proof}
Assume that for some infinite spaces $Y$ and $Z$ there exists a continuous linear  surjection $T:C_{p}(X)\rightarrow C_{p}(Y\times Z)$. Then the space $Y\times Z$ is pseudocompact, see Remark \ref{remark_pseudo}.
By Theorem \ref{solution}  $C_{p}(Y\times Z)$ contains a complemented copy of $(c_{0})_{p}$. Hence $C_{p}(X)$  maps onto $(c_{0})_{p}$ by a continuous linear map, and then  by  Theorem \ref{ba-ka1} the space $C_{p}(X)$ has a complemented copy of $(c_{0})_{p}$. This  implies that  $C(\beta X)$ contains a complemented copy of $c_{0}$, a contradiction with the Grothendieck property of $C(\beta X)$.
\end{proof}

For the proof of Corollary \ref{separable} we will need some auxiliary
facts.

The next proposition is known; the case of the pointwise topology
can be found in \cite[Exercise 152]{Tka1} and the general result covering the cases
of both topologies is stated in \cite[Exercise 2, p.\ 36]{mccoy},
but without a proof, hence, for the sake of completeness, we include
the proof for the compact-open topology. Recall that a subset $Y$
of a topological space $X$ is \emph{C-embedded} (\emph{C$^*$-embedded}) in $X$ if every (bounded)
continuous real-valued function on $Y$ can be continuously extended
over $X$. Given a compact subspace $K$ of a space $X$ and
$\varepsilon>0$, we denote the set $\{f\in C_k(X): f(K)\subset
(-\varepsilon,\varepsilon)\}$, a basic neighborhood of zero in
$C_k(X)$, by $U_X(K,\varepsilon)$.

\begin{proposition}\label{restriction_open}
Let $Y$ be a subspace of a topological space $X$, and $R: C_k(X)\to C_k(Y)$ ($R: C_p(X)\to C_p(Y)$) be the restriction operator defined by $R(f) = f\rstr Y$. $R$ is open  if and only if $Y$ is closed and C-embedded in $X$.
\end{proposition}

\begin{proof}
As we mentioned, the proof for the case of the pointwise topology can be found in \cite[Exercise 152]{Tka1}, therefore we will only give the proof for the case of the compact-open topology.

Assume first that $R$ is open. Then obviously it is a surjection, which means that $Y$ is  C-embedded in $X$. Suppose, towards a contradiction, that $Y$ is not closed in $X$ and pick a point $x\in \overline{Y}\setminus Y$. We will show that the image $R(U_X(\{x\},1))$ is not open in $C_k(Y)$. Consider an arbitrary basic neighborhood of zero $U_Y(K,\varepsilon)$ in $C_k(Y)$, given by a compact set $K\subset Y$ and $\varepsilon>0$. The set $L=\{x\}\cup K$ is compact in $X$, hence it is C-embedded in $X$, see \cite[3.11]{gillman}.

Let $f:L\to \mathbb{R}$ be the (continuous) function which takes value $1$ at $x$ and value $0$ on $K$, and let $F$ be its continuous extension over $X$. Put $g = F\rstr Y$. Obviously, for any $h\in C_k(X)$ such that $h\rstr Y = g\rstr Y$, the functions $h$ and $F$ must agree on the closure of $Y$, so $h(x)=F(x)=1$. This shows that $$g\in U_Y(K,\varepsilon)\setminus R(U_X(\{x\},1)).$$

Now, assume that $Y$ is closed and C-embedded in $X$. To prove that $R$ is open its enough to verify that, for any compact $K\subset X$ and $\varepsilon>0$, we have $$R(U_X(K,\varepsilon)) = U_Y(K\cap Y,\varepsilon).$$ In the last equality, one inclusion is obvious. Take any $$g\in U_Y(K\cap Y,\varepsilon).$$ We will show that $g$ is an image under $R$ of some $f\in U_X(K,\varepsilon)$. Let $G$ be a continuous extension of $g$ over $X$. Put $$L = \{x\in K: |G(x)|\ge \varepsilon\}.$$ Clearly, $L$ is a compact subset of $X$ disjoint with $Y$. Hence we can use the well known fact (see \cite[3.11]{gillman}), that a compact set can be separated from a closed set by a continuous function, i.e., we can find a continuous $h:X\to [0,1]$ which sends $L$ to $0$ and $Y$ to $1$. One can easily verify that $f = hG$ has the required properties.
\end{proof}

Recall that if $X$ is a pseudocompact space  and $K$ is  a  compact space, then $X\times K$ is pseudocompact (see \cite[Corolllary 3.10.27]{En}), so $C_{p}(X\times K)$ contains a complemented copy of $(c_{0})_{p}$, provided that both $X$ and $K$ are infinite. For the compact-open topology in this function space we have the following special case of Corollary \ref{separable}.
\begin{proposition}\label{lastcase}  Assume that $X$ is an   infinite pseudocompact space  and $K$ is  an infinite compact space.
\begin{enumerate}
\item If $X$ contains an infinite compact subset, then  $C_{k}(X\times K)$ contains a complemented copy of the Banach space $c_{0}$.
\item If $X$ has no infinite compact subsets, then  $C_{k}(X\times K)$ contains a complemented copy of  $(c_{0})_{p}$.
\end{enumerate}
\end{proposition}
\begin{proof}
Part (1): By Doma\'nski's and Drewnowski's theorem \cite{drew}, the space $C_{k}(X,C_{k}(K))$ contains a complemented copy of $c_{0}$. On the other hand, by \cite[Corollary 2.5.7]{mccoy} the spaces $C_{k}(X\times K)$ and  $C_{k}(X,C_{k}(K))$ are linearly homeomorphic.

Part (2): Since the product $X\times K$ is pseudocompact, we can apply Theorem \ref{pseudocompact_product_has_JNP}
to obtain a \jns\ $(\mu_{n})_{n\in\omega}$ on $X\times K$ such that $(\mu_{n})_{n\in\omega}$ is supported on the product $D\times E$, where $D\subset X$ and $E\subset Y$ are countable discrete, and the supports of $\mu_n$ have pairwise disjoint projections onto each axis.
Since the sets $D$ and $E$ are discrete and countable, we can find families $\{U_d: d\in D\}$ and $\{V_e: e\in E\}$ of  pairwise  disjoint sets, such that, for each $d\in D$, $U_d$ is an open neighborhood of $d$ in $X$ and for each $e\in E$, $V_e$ is an open neighborhood of $e$ in $Y$.
Let $A_n = \supp(\mu_n)$ and $A = \bigcup_{n\in\omega}A_n$. Given $a = (d_a,e_a)\in A$,  put $W_a = U_{d_a}\times V_{e_a}$. Clearly, the family $\{W_a: a\in A\}$ of neighborhoods of points of $A$ consists of pairwise  disjoint sets. Moreover, if we define, for each $n\in \omega$, $$W_n = \bigcup\{W_a: a\in A_n\},$$ then the properties of supports $A_n$ imply that the sets $W_n$ have pairwise disjoint projections onto each axis.  For each $a\in A$, take a continuous function $$g_a: X\times K\to [0,1]$$ such that $g_a(a)=1$ and $g$ takes value $0$ on $(X\times K)\setminus W_a$. For every $n\in \omega$ define $$f_n: X\times K\to [-1,1]$$ by
$$f_n = \sum_{a\in A_n}\frac{\mu_n(a)}{|\mu_n(a)|}g_a\,.$$

The functions $f_n$ and the sets $W_n$ have the following properties for all $n\in\omega$:
\begin{enumerate}[(a)]
\item $\mu_n(f_n) = 1$;
\item the support of $f_n$ is contained in $W_n$;
\item the support of $\mu_n$ is contained in $W_n$;
\item the projections of the sets $W_k,\ k\in\omega$, onto $X$ are pairwise disjoint.
\end{enumerate}

Consider the linear operator $$S:(c_{0})_{p}\to C_{k}(X\times K)$$
defined by $$S((t_n)) = \sum_{n\in\omega} t_nf_n$$ for $(t_n)\in
(c_{0})_{p}$. By properties (b) and (d),  the supports of $f_n$ are
pairwise disjoint, therefore $S$ is well defined. Given a compact
subset $L$ of $X\times K$, the projection of $L$ onto $X$ is finite,
hence, by properties (b) and (d), $L$ intersects only finitely many
supports of $f_n$, say only for $n$ in some finite set $F$.
Therefore, for any $\varepsilon>0$, if $|t_n|<\varepsilon$, for
$n\in F$, then  $S((t_n))\in U_{X\times K}(L,\varepsilon)$, which
means that the operator $S$ is continuous. Put $Z = S((c_{0})_{p})$.
We will show that $Z$ is an isomorphic copy of $(c_{0})_{p}$ which
is complemented in $C_{k}(X\times K)$.

Let $T: C_{k}(X\times K)\to (c_{0})_{p}$ be defined by $T(f)(n) = \mu_n(f)$ for $f\in C_{k}(X\times K)$ and $n\in\omega$. Obviously, the operator $T$ is continuous. Using properties (a)--(d) one can easily verify that $$(T\rstr Z)\circ S = \mathrm{Id}_{(c_{0})_{p}}$$ and $S\circ(T\rstr Z) = \mathrm{Id}_{Z}$, hence the spaces $Z$ and $(c_{0})_{p}$ are isomorphic. Let $$P = S\circ T: C_{k}(X\times K)\to C_{k}(X\times K).$$ The identities $Z = S((c_{0})_{p})$ and $S\circ(T\rstr Z) = \mathrm{Id}_{Z}$ imply that $P$ is a continuous projection of $C_{k}(X\times K)$ onto $Z$.
\end{proof}

We are ready to provide a proof of Corollary \ref{separable}.

\begin{proof}[{\normalfont\textbf{Proof of Corollary \ref{separable}}}]
Part (1) is a direct consequence of Theorem \ref{solution}.

For the proof of part (2) we shall consider several cases:
\smallskip\\
{\bf Case 1.}
If both spaces $X$ and $Y$ contain infinite compact subsets, say $K$ and $L$, respectively, then $K\times L$ is C-embedded in $X\times Y$ (see \cite[Section 3.11]{gillman}), and the restriction map $C_k(X\times Y)\rightarrow C_{k}(K\times L)$ is a continuous and open surjection by Proposition \ref{restriction_open}. Now we can use Proposition \ref{lastcase}.(1). %Theorem \ref{cembranos}.(2).
\smallskip\\
{\bf Case 2.}
If neither $X$ nor $Y$ admits an infinite compact subset, then  part (1) applies for $C_{p}(X\times Y)=C_{k}(X\times Y)$.
\smallskip\\
{\bf Case 3.}
If one of the spaces  $X$ or $Y$ is not pseudocompact, the product $X\times Y$ is also not pseudocompact, and the space $C_{k}(X\times Y)$ contains a complemented copy of  $\mathbb{R}^\omega$, see \cite[Theorem 14]{ka-sa}.
\smallskip\\
{\bf Case 4.}
 The only remaining case to be checked is that  when both $X$ and $Y$ are pseudocompact, and one of these spaces contains an infinite compact subset but the other one lacks infinite compact subsets. Without loss of generality we can assume that $Y$ contains an infinite compact subset $K$. We will verify that $X\times K$ is C-embedded in $X\times Y$. Let $f$ be a continuous function on $X\times K$. Since $X\times K$ is pseudocompact, by Glicksberg's theorem \cite[Exercise 3.12.21(c)]{En}, $f$ can be extended to a continuous function $g$ on $\beta X\times K$. By compactness of $\beta X\times K$, we can extend $g$ continuously over  $\beta X\times Y$ obtaining a function $h$. Now, the restriction of $h$ to $X\times K$ is the desired extension of $f$. Proposition \ref{restriction_open} implies that the restriction operator $R: C_{k}(X\times Y)\to  C_{k}(X\times K)$ is open, hence also surjective. From Proposition \ref{lastcase} we obtain a projection $P$ of $C_{k}(X\times K)$ onto a subspace $Z$ isomorphic to $(c_{0})_{p}$. Since $P$ is open (as a map onto $Z$), the composition $P\circ R$ is an open continuous surjection of $C_{k}(X\times Y)$ onto a copy of $(c_{0})_{p}$.
\end{proof}

 \section{Haydon's construction of a pseudocompact space with no infinite compact subsets}

Theorem \ref{solution} may suggest  a question whether  $C_{p}(X\times X)$ contains a complemented copy of $(c_{0})_{p}$ for any infinite pseudocompact space $X$. In Section \ref{section_x_x_no_bjns}, assuming the Continuum Hypothesis (or even a weaker set-theoretic assumption), we will answer this question negatively. Our examples use the following general scheme of constructing pseudocompact spaces with no infinite compact subsets given by Haydon in \cite{Ha}.

For each  $A\in [\omega]^\omega$,  choose an ultrafilter $u_A\in\omega^*$ in the closure of $A$ in $\beta\omega$. Let $$X = \omega\cup\{u_A: A\in [\omega]^\omega\}$$ be topologized as a subspace of $\beta\omega$. To simplify the notation we will denote the family of all such spaces $X$ by \hs.
\medskip

We start with the following fact collecting some properties of every space $X\in\hs$.
\begin{proposition}\label{HS_properties}
Every space $X\in\hs$ has the following properties:
\begin{enumerate}
\item $X$ is pseudocompact of cardinality continuum;
\item all compact subspaces of $X$ are finite;
\item $C_p(X)$ does not have the \JNP;
\item  $C_{p}(X)$ admits an infinite dimensional metrizable quotient isomorphic to the subspace $(\ell_\infty)_{p}=\{(x_n)\in \IR^\omega:\sup_n |x_n|<\infty\}$ of $\IR^\omega$ endowed with the product topology and the Banach space $(C(X),\|.\|_{\infty})$ (i.e. endowed with the sup-norm topology) is isometric  to the Banach space $\ell_{\infty}$;
\item  for every infinite compact space $K$ both spaces $C_{p}(K\times X)$ and $C_{k}(K\times X)$ contain a  complemented copy of $(c_{0})_{p}$.
\end{enumerate}
\end{proposition}
\begin{proof}
 Obviously $|X|\le 2^\omega$. To see that $X$ must contain continuum many ultrafilters, it is enough to take an almost disjoint family $\mathcal{A}\subset  [\omega]^\omega$ (i.e., distinct members of $\mathcal{A}$ have finite intersection) of cardinality continuum. Then the closures of elements of $\mathcal{A}$ in $\beta\omega$ have pairwise disjoint intersections with $\omega^*$. This proves the second item of (1).

The first item of  (1) has been shown by Haydon in the third example of \cite{Ha}. % To keep the proposition self-contained we add a  short proof of it due to Haydon. We show that $X$ is pseudocompact.
Indeed, if $X$ is not pseudocompact, then there exists $f\in C(X)$ unbounded.  As $\omega$ is dense in $X$ one gets a sequence $(n_{k})$ in $\omega$ such that $|f(n_{k})|>k$. For the set $A = \{{n_{k}:k\in\omega}\}$ the corresponding $u_{A}$ in $\beta\omega$ belongs to the closure in $\beta\omega$ of $\{n_{k}:k\geqslant m\}$ for each $m\in\omega$. Hence $|f(u_{A})|>m$ for each $m\in\omega$, a contradiction.

Proof of (2): Since infinite compact subsets of $\beta\omega$  have cardinality $2^{\mathfrak{c}}$, every compact subset of $X$ is finite by (1).

Proof of (3): Note that $\beta X=\beta\omega$ and the Banach space $(C(X),\|.\|_{\infty})$ is isometric to $\ell_{\infty}= C(\beta\omega)$.  On the other hand, as $C(\beta\omega)$ is a Grothendieck space,  $C(X)$ does not have a complemented copy of the Banach space $c_{0}$. Consequently, by applying the closed graph theorem \cite[Theorem 4.1.10]{bonet},  the space $C_{p}(X)$ does not contain a complemented copy of $(c_{0})_{p}$. Now by   Theorem \ref{ba-ka1} we know that $C_p(X)$ does not have \JNP.

Proof of (4):  Since the space $X$ is pseudocompact and contains discrete $\omega$ which is C$^{*}$-embedded into $X$, we apply Theorem 1 from  \cite{BKS} to deduce that $C_{p}(X)$ has a quotient $C_{p}(X)/W$ isomorphic to the subspace $(\ell_\infty)_{p}$ of $\IR^\omega$ endowed with the product topology, where $$W=\bigcap_{n} \{f\in C_p(X): \sum_{x\in F_n} f(x)=0 \}$$ and $(F_{n})_{n\in\omega}$ is any sequence  of non-empty, finite and pairwise disjoint subsets of $\omega$ with $\lim_n |F_n|=\infty$.
The  second part of the item (4) is clear since $C(\beta\omega)$ and  $(C(X),\|.\|_{\infty})$ are isometric.

Proof of (5): Since $K\times X$ is pseudocompact, the first claim of (5) follows from Theorem \ref{solution}. The other claim follows from Proposition \ref{lastcase}.
\end{proof}

The following proposition characterizes those subspaces $X$ of $\beta\omega$ containing $\omega$ which are in the class \hs. %We will need also the following observation: The above family \hs\ can be described without referring to the assignment $A\mapsto u_A$.

\begin{proposition}\label{charact_HS}
For a subspace $X$ of $\beta\omega$ containing $\omega$, the following conditions are equivalent:
\begin{enumerate}
\item $X\in\hs$;
\item $X$ is pseudocompact and of cardinality continuum;
\item $X$ is of cardinality continuum, and every infinite subset of $\omega$ has an accumulation point in $X$.
\end{enumerate}
\end{proposition}
\begin{proof}
The implication (1)$\Rightarrow$(2) was explained in the previous proposition.

The implication (2)$\Rightarrow$(3) is obvious, a pseudocompact $X$ cannot contain an infinite clopen discrete subset.

To verify (3)$\Rightarrow$(1), first observe that each ultrafilter in $X\cap \omega^*$ is in the closure of continuum many sets $A\in [\omega]^\omega$. Also, given $A\in [\omega]^\omega$, the closure of $A$ in $X$ has cardinality continuum, which is witnessed by an almost disjoint family $\mathcal{A}$ of subsets of $A$ of cardinality $|\mathcal{A}| = 2^\omega$. Therefore, by a standard back-and-forth inductive argument, we can construct a bijection $f$ between $[\omega]^\omega$ and $X\cap \omega^*$ such that $u_A = f(A)$ is in the closure of $A$. To this end, enumerate $[\omega]^\omega$ as $\{A_\alpha: \alpha < 2^\omega\}$, and $X\cap \omega^*$ as $\{u_\alpha: \alpha < 2^\omega\}$. Inductively, for each $\alpha < 2^\omega$, define $f$ on  $\{A_\beta: \beta < \alpha\}$ and $f^{-1}$ on  $\{u_\beta: \beta < \alpha\}$.
\end{proof}
In \cite{marci} Marciszewski  constructed an example of a space $X\in\hs$ such that there is no
continuous linear surjection from the space $C_p(X)$ onto $C_p(X)\times\mathbb{R}$. In particular,
$C_p(X)$ is not linearly
homeomorphic to $C_p(X)\times\mathbb{R}$, hence $C_p(X)$ does not contain any complemented copy of $c_0$ or $\mathbb{R}^\omega$. To achieve this it is enough to require that all ultrafilters $u_A$ in $X$ are weak $P$-points and they are pairwise non-isomorphic.
% Proposition \ref{not-pseudocompact} below   shows that in this case the square $X\times X$ is not pseudocompact, therefore $C_p(X\times X)$ has a complemented copy of $\mathbb{R}^\omega$.

Recall that a point $p\in\omega^*$ is a {\it weak $P$-point} if
$p$ is not in the closure of any countable set
$D\subset \omega^*\setminus\{p\}$, see
\cite{Mill}. Two ultrafilters $u,v\in\beta\omega$ are called {\em isomorphic} if $h(u)=v$ for some homeomorphism $h$ of $\beta\omega$; equivalently, there is a bijection $f: \omega\to \omega$ such that $u = \{f^{-1}(A): A\in v\}$. We will appeal to weak $P$-points in Section \ref{section_x_x_no_bjns}.

Since there are $\mathfrak{c}$ many bijections $f:\omega\to\omega$ and $2^\mathfrak{c}$ many ultrafilters on $\omega$, for each two sets $A,B\in[\omega]^\omega$ one can easily find two non-isomorphic ultrafilters $u_A\in\clo_{\beta\omega}(A)$ and $u_B\in\clo_{\beta\omega}(B)$. This observation will be useful in the next proposition.

\begin{proposition}\label{not-pseudocompact}
Let $X = \omega\cup\{u_A: A\in [\omega]^\omega\}$ be a space in \hs\ such that for distinct $A,B\in [\omega]^\omega$ the ultrafilters $u_A,u_B$ are not isomorphic. Then the square $X\times X$ is not pseudocompact.
\end{proposition}
\begin{proof}
Take two disjoint infinite sets $A,B\subset\omega$ and any bijection $f:A\to B$. Observe that the graph $\Gamma=\{(x,f(x)):x\in A\}$ of $f$ is an open discrete subspace of $X\times X$. We will show that $\Gamma$ is closed in $X\times X$.

Suppose,  towards a contradiction that, there exists $(u_1,u_2)\in
\clo_{X\times X}(\Gamma) \setminus \Gamma$. For any $n\in\omega$,
the intersections $\Gamma\cap (\{n\}\times X)$ and $\Gamma\cap(X
\times \{n\})$ contain at most one point. Therefore, they are closed
in $\{n\}\times X$ and $X \times \{n\}$, respectively. Since the
latter sets are open in $X\times X$, it follows that
$u_1,u_2\in\omega^*$. Since $\Gamma\subset A\times B$, we have
$$(u_1,u_2)\in \clo_{X\times X}(A\times B) = \clo_{X}(A)\times
\clo_{X}(B).$$ The disjoint sets $A$ and $B$ have disjoint closures
in $X$, so $u_1\neq u_2$.

Let $g:\omega\to \omega$ be  any bijection extending $f$.  We claim that
the ultrafilter $\{g^{-1}(C): C\in u_2\}$ is equal to $u_1$. Suppose
that these two ultrafilters are not equal. Then we can find
$D\subset \omega$ such that $D\in u_1$ and $$(\omega\setminus\ D)\in
\{g^{-1}(C): C\in u_2\},$$ i.e., $E = g(\omega\setminus\ D)\in u_2$.
Observe that the product $D\times E$ is disjoint with the graph of
$g$, hence is also disjoint with $\Gamma$. Therefore the product
$\clo_{X}(D)\times \clo_{X}(E)$ is a neighborhood of $(u_1,u_2)$
disjoint with $\Gamma$ (because $\Gamma\subset\omega\times\omega$
and $$(\clo_{X}(D)\times \clo_{X}(E))\cap\omega\times\omega=D\times
E,$$ and the latter is disjoint with $\Gamma$), a contradiction. We
obtain that the ultrafilters $u_1$ and $u_2$ are isomorphic,
contradicting our assumption on $X$, and thus showing that $\Gamma$ is closed.

Since the product $X\times X$ contains an infinite clopen discrete subspace $\Gamma$, it cannot be pseudocompact.

%Finally, we verified that the infinite discrete subspace $\Gamma$ is clopen in $X\times X$, hence the space $X\times X$ is not pseudocompact.
\end{proof}
On the other hand, we have the following example.
\begin{example}\label{second}
There exists a space $Y\in\hs$ such that the square $Y\times Y$ is pseudocompact. Consequently,
 $C_{p}(Y\times Y)$ contains a complemented copy of $(c_{0})_{p}$.
\end{example}
\begin{proof}
For any space $X\in\hs$, the square $\omega\times\omega$ is dense in $X\times X$, therefore it is enough to construct a space $Y\in\hs$ such that any infinite subset of $\omega\times\omega$ has an accumulation point in $Y\times Y$, see the proof of Proposition \ref{HS_properties}.(1).

Let $\{A_\alpha: \alpha < 2^\omega\}$ be an enumeration of $[\omega]^\omega$, and $\{C_\alpha: \alpha < 2^\omega\}$  be an enumeration of $[\omega\times\omega]^\omega$. For each $\alpha < 2^\omega$, choose an ultrafilter $u_{\alpha}$  from $\omega^*$ in the closure of $A_\alpha$,
and an accumulation point $(p_\alpha,q_\alpha)$ of $C_\alpha$ in $\beta\omega\times\beta\omega$. Let $$Y= \omega\cup \{u_\alpha,p_\alpha,q_\alpha: \alpha < 2^\omega\}.$$ One can easily verify that $Y$ satisfies condition (3) from Proposition \ref{charact_HS} (note that the points $u_\alpha$'s are added to $Y$ in order to ensure that every infinite subset of $\omega$ has an accumulation point in $Y$), so it is in \hs, and  any infinite subset of $\omega\times\omega$ has an accumulation point in $Y\times Y$. The last claim follows  from Theorem \ref{solution}.
\end{proof}

If $X$ is the space from Proposition \ref{not-pseudocompact} and $Y$ as in Example \ref{second}, then by taking the disjoint union $Z=X\sqcup Y$ we get the following corollary.

\begin{corollary}
There exists a pseudocompact space $Z$ such that its square $Z\times Z$ is not pseudocompact, but $C_p(Z\times Z)$ contains a complemented copy of $(c_0)_p$.
\end{corollary}

\section{The Bounded Josefson–Nissenzweig Property}

In this section we define a ``bounded'' version of the Josefson–Nissenzweig property. A sequence $(\mu_n)$ of finitely supported signed measures  on $X$ such that $\|\mu_n\|=1$ for all $n\in \mathbb{N}$, and $\lim\limits_{n}\mu_n(f) = 0$ for each bounded $f\in C_p(X)$ is called {\em a bounded Josefson--Nissenzweig sequences} or, in short, {\em a \bjns}. We say that $C_{p}(X)$  has {\em the Bounded Josefson--Nissenzweig Property} or, in short, {\em the \BJNP}, if $X$ admits a \bjns.

Obviously, for pseudocompact spaces $X$, these bounded versions coincide with the standard ones. It is also trivial that the  \JNP\ implies the \BJNP\ and each \jns\ is a \bjns.

One can easily construct examples of \bjns s, which are not \jns s. Actually, if  $X$ is not pseudocompact and  $C_{p}(X)$ has the \JNP, then $X$ has a  \bjns\, which is not a \jns. Indeed, if $X$ is such a space, take a \jns\ $(\mu_n)_n$ on $X$ and a continuous unbounded function $f$ on X. For each $n$, pick $x_n\in X$ such that
$|f(x_n)| > n$. Define $\nu_n = \mu_n + (1/n)\delta_{x_n}$. One can easily verify that the sequence $(\nu_n)_n$ (after normalizing) is a \bjns\ but not a \jns.

%We also have the following

In Example \ref{bjnp_not_jnp} we construct a space with the BJNP but without the JNP. For a verification of the properties of the example we will need an auxiliary fact concerning \jns s. Recall that a subset $A$ of a topological space $X$ is \textit{bounded} if for every $f\in C_p(X)$ the set $f(A)$ is bounded in $\mathbb{R}$.

\begin{proposition}\label{union_supp_bounded}
Let $(\mu_n)_{n\in\omega}$ be a \jns\ of measures on a space $X$. Then the union $\bigcup_{n\in\omega}\supp(\mu_n)$ of supports of $\mu_n$ is bounded in $X$.
\end{proposition}
\begin{proof}
We sketch the proof of the proposition using operators between linear spaces. In \cite[Lemma 4.11]{KSZ} we provide a more direct and elementary proof.

Let $S = \{0\}\cup\{1/n: n=1,2,\dots\}$ be equipped with the Euclidean topology. We can define a continuous linear operator $T: C_p(X)\to C_p(S)$ in the following way: for $f\in C_p(X)$, put $T(f)(0) = 0$ and $$T(f)(1/(n+1)) = \mu_n(f)$$ for $n \in\omega$.
From the definition of a \jns\ immediately follows that the operator $T$ is well-defined. Observe that, for any $n\ge 1$, the support $\supp_T(1/n)$ of $1/n$  in $X$ with respect to $T$ (see \cite[Chapter 6.8]{vM2}) is equal to $\supp(\mu_n)$. Obviously the set $A= \{1/n: n\ge1\}$ is bounded in $S$, $S$ being compact. By Theorem 6.8.3 in \cite{vM2}, the support  $$\supp_T(A) = \bigcup_{n\ge 1}\supp_T(1/n) = \bigcup_{n\in\omega}\supp(\mu_n)$$ is bounded in $X$.
\end{proof}

\begin{example}\label{bjnp_not_jnp}
There exists a countable space $X$ such that $C_{p}(X)$  has the \BJNP, but does not have the \JNP.
\end{example}

\begin{proof}%[Construction of Example \ref{bjnp_not_jnp}]
Recall the well known method of obtaining  a countable space $N_F$ associated with a filter $F$ on a countable set $T$ (we consider only free filters
on $T$, i.e., filters containing all cofinite subsets of $T$).

$N_F$ is the space $T\cup\{\infty\}$, where $\infty\not\in T$, equipped with the following
topology:
All points of $T$ are isolated and the family
$\{A\cup\{\infty\}: A\in F\}$ is
a neighborhood base at $\infty$. One can easily observe that basic open sets are also closed, therefore the space $N_F$ is Tychonoff.

Let $F$ be the filter on $\omega$ consisting of sets of density $1$, i.e.,
$$F= \Big\{A\subset\omega: \lim_{n}\frac{|A\cap\{1,2,\dots,n\}|}{n} = 1\Big\}\,.$$
We will show that the space $X= N_F$ has the required properties (this space in the context of function spaces was investigated in \cite{DMM}).

Let $S_n = \{2^{n}+1,2^{n}+2,\dots,2^{n+1}\},\ n\in\omega$, and

\smallskip

$$\lambda_n = \sum\left\{\frac{1}{2^{n+1}}\delta_k: k\in S_n\right\} - \frac{1}{2}\delta_\infty\,.$$

\smallskip

Let us check that $(\lambda_n)_{n\in\omega}$ is a \bjns. Take a bounded continuous function $f$ on $N_F$; for simplicity, we can assume that $|f|\le 1$. Fix an $\varepsilon > 0$, and find a set $A\in F$ such that the image under $f$ of the neighborhood $A\cup\{\infty\}$ of $\infty$ has the diameter less than $\varepsilon$. Choose $k\in\omega$ such that, $$\frac{|A\cap\{1,2,\dots,m\}|}{m}> 1 - \varepsilon,\quad  \mbox{for } m > k\,.$$
Observe that, for $n > k$, we have $2^{n+1} > k$, $|S_n|=2^n$ and $S_n\subset  \{1,2,\dots,2^{n+1}\}$, hence the above inequality applied for $m = 2^{n+1}$, gives us the estimate $|S_n\setminus A| < 2^{n+1}\varepsilon$. Hence, for $n > k$, we have
\begin{eqnarray*}
|\lambda_n(f)| &=& \big|\sum\left\{\frac{1}{2^{n+1}}f(k): k\in S_n\right\} - \frac{1}{2}f(\infty)\big| \le \\ &\le& \sum\left\{\frac{1}{2^{n+1}}|f(k) - f(\infty)|: k\in S_n\right\} =\\ &=& \sum\left\{\frac{1}{2^{n+1}}|f(k) - f(\infty)|: k\in S_n\cap A\right\} +\\ &+& \sum\left\{\frac{1}{2^{n+1}}|f(k) - f(\infty)|: k\in S_n\setminus A\right\}<\\ &<& |S_n\cap A|\cdot\frac{\varepsilon}{2^{n+1}} + |S_n\setminus A|\cdot\frac{2}{2^{n+1}} < 2^n\cdot\frac{\varepsilon}{2^{n+1}} + 2^{n+1}\varepsilon\cdot\frac{2}{2^{n+1}} = \frac52\varepsilon\,,
\end{eqnarray*}
which shows that the sequence $(\lambda_n(f))$  converges to $0$.

It is clear that for any \jns\ $(\mu_n)$ on a space $X$, the union $S$ of supports of $\mu_n$ is infinite, and by Proposition \ref{union_supp_bounded} $S$ is also bounded. Since all bounded subsets of $N_F$ are finite (cf.\ \cite[Example 7.1]{DMM}), $C_{p}(N_F)$ does not have the \JNP.\end{proof}

Recall that $C_p^\ast(X)$ is the subspace of $C_p(X)$ consisting of bounded functions.

\begin{proposition}\label{bjnp_beta_c_0}
For every space $X$, if $C_{p}(X)$ has the \BJNP, then $C_p(\beta X)$ has the JNP and $C(\beta X)$ contains a complemented copy of the Banach space $c_{0}$.
\end{proposition}
\begin{proof}
Assume that a space $X$ is such that $C_p(X)$ has the \BJNP. Every \bjns\ on a space $X$ is trivially a \jns\ on $\beta X$, so $C_p(\beta X)$ has \JNP\ and thus contains a complemented copy of $(c_0)_p$ by Theorem \ref{ba-ka1}. By \cite[Proposition 6.1]{KSZ} the space $C(\beta X)$ contains a complemented copy of $c_0$.
\end{proof}

Note that the above proposition remains true if we replace $\beta X$ by any compact space $K$ containing $X$.

\begin{theorem}\label{bjnp_c_0}
For a space $X$ the following conditions are equivalent:
\begin{enumerate}
\item $C_{p}(X)$ has the \BJNP;
\item $C_p^\ast(X)$ has a complemented copy of $(c_0)_{p}$;
\item $C_p^\ast(X)$ has a quotient isomorphic to $(c_0)_{p}$;
\item $C_p^\ast(X)$ admits a continuous linear surjection onto $(c_0)_{p}$.
\end{enumerate}
\end{theorem}
\begin{proof}
The implication (1)$\Rightarrow$(2) follows from Lemma 1 in \cite{BKS1} and its proof.  It is enough to apply this lemma for $K= \beta X$, and observe that the subspace $L$ of $C_p(X)$ given by this lemma contains $C_p^\ast(X)$ and the range of the projection $P$ from $L$ onto the subspace isomorphic to $(c_0)_p$ is contained in $C_p^\ast(X)$.

Implications (2)$\Rightarrow$(3)$\Rightarrow$(4) are obvious.

(4)$\Rightarrow$(1). Let $T: C_p^\ast(X)\to (c_0)_{p}$ be a continuous linear surjection. For each $n\in\omega$, let $p_n: (c_0)_{p}\to \mathbb{R}$ be the projection onto the $n$th axis, and $\lambda_n$ be the finitely supported signed measure on $X$ corresponding to the functional $p_n\circ T$. Let $C^\ast(X)$ be the Banach space of all bounded continuous functions on $X$, equipped with the standard supremum norm. By the Closed Graph Theorem, $T$ can be treated as a continuous linear surjection between the Banach spaces $C^\ast(X)$ and $c_0$, which gives the estimate $\|\lambda_n\|\le \|T\|$ for all $n$. By the Open Mapping Theorem we can also get a constant $c>0$ such that $\|\lambda_n\|\ge c$ for all $n$. Now it is clear that $(\lambda_n/\|\lambda_n\|)_{n\in\omega}$ is a \bjns\ on $X$.
\end{proof}

\begin{corollary}\label{corollary_cpstar_no_c0_cp_no_c0}
For every space $X$, if $C_p^*(X)$ does not have any complemented copy of $(c_0)_p$, then $C_p(X)$ does not have it either.
\end{corollary}
\begin{proof}
If $C_p(X)$ does have a complemented copy of $(c_0)_p$, then by Theorem \ref{ba-ka1} the space $C_p(X)$ has the \JNP\  and so it has the \BJNP, which, by Theorem \ref{bjnp_c_0}, implies that $C_p^*(X)$ has a complemented copy of $(c_0)_p$.
\end{proof}

%Example \ref{bjnp_not_jnp} combined with Theorems \ref{ba-ka1} and \ref{bjnp_c_0} gives us the following

\begin{corollary}
There exists a countable space $X$ such that:
\begin{enumerate}
	\item $C_p^\ast(X)$ has a complemented copy of $(c_0)_{p}$,
	\item $C_p(X)$ does not have such a copy,
	\item the Banach space $C(\beta X)$ contains a complemented copy  of the Banach space $c_{0}$.
\end{enumerate}
\end{corollary}
\begin{proof}
Let $X$ be the countable space described in Example \ref{bjnp_not_jnp}. The space $C_p(X)$ has the \BJNP, so by Theorem \ref{bjnp_c_0} the space $C_p^*(X)$ has a complemented copy of $(c_0)_p$ and by Proposition \ref{bjnp_beta_c_0} the space $C(\beta X)$ has a complemented copy of $c_0$. On the other hand, $C_p(X)$ does not have the \JNP, which by Theorem \ref{ba-ka1} implies that it does not have any complemented copies of $(c_0)_p$.
\end{proof}

\section{A pseudocompact space $X$ such that its square $X\times X$ does not have a \bjns}\label{section_x_x_no_bjns}

First we will prove several auxiliary results. The proof of the
following standard fact is similar to that of \cite[Lemma
3.1]{marci}.

\begin{lemma}\label{lemma_Ma} Let
$\{S_n: n\in \omega\}$ be a family of pairwise disjoint  subsets of
$\omega$. Then, for every subset $C$ of $\beta\omega$ of the
cardinality less than $2^\omega$, there exists an infinite subset
$A\subset \omega$ such that $$C\cap\big(\clo_{\beta\omega}\bigcup_{n\in A}S_n\big)\subset\bigcup_{n\in A}\clo_{\beta\omega}S_n.$$
\end{lemma}
\begin{proof}
Let $\mathcal A\subset [\omega]^\omega$ be an almost disjoint family
of size $|\mathcal A|=\mathfrak c$. It is clear that
$$ \big(\clo_{\beta\omega}\bigcup_{n\in A}S_n\big) \:\cap\: \big(\clo_{\beta\omega}\bigcup_{n\in A'}S_n\big)=
\bigcup_{n\in A\cap A'}\clo_{\beta\omega}S_n$$ for any distinct
$A,A'\in\mathcal A$. Thus the family
$$ \big\{\big(\clo_{\beta\omega}\bigcup_{n\in A}S_n\big)\setminus \bigcup_{n\in A}\clo_{\beta\omega}S_n\: :\: A\in\mathcal A\big\} $$
is disjoint, and hence one of its elements is disjoint from $C$ as
the latter has size $<\mathfrak c=|\mathcal A|$. This completes our
proof.
\end{proof}

Recall that {\em the Rudin--Keisler preorder} on  $\beta\omega$ is the binary
relation $\le_{\rk}$ on $\beta\omega$ given by $v\le_{\rk}u$ if
there is $f: \omega\to \omega$ such that $u \supset \{f^{-1}(A):
A\in v\}$, i.e., $v=f(u)$, where $f(u)=\{A\subset\omega: f^{-1}(A)\in
u\}$. This preorder becomes an order if we identify isomorphic
ultrafilters, see \cite[Corollary 9.3]{CN}. We say that ultrafilters
$u,v\in\omega^*$ are {\em incompatible} with respect to the
Rudin--Keisler preorder if there is no $w\in\omega^*$ such that
$w\le_{\rk}u$ and $w\le_{\rk}v$. An {\em \rk-antichain} in
$\omega^*$ is a subset of $\omega^*$ consisting of pairwise
incompatible ultrafilters  with respect to the Rudin--Keisler
preorder. The Continuum Hypothesis or even a weaker set-thoretic
assumptions, like Martin's Axiom, imply that there exists an
\rk-antichain of size continuum, consisting of weak $P$-points in
$\omega^*$, cf. \cite[Theorem 9.13, Lemma 9.14]{CN} or
\cite[Theorem 2]{Bl}. Blass and Shelah \cite{BS} proved the
consistency, relative to \textsf{ZFC}, of the statement that any two
ultrafilters in $\omega^*$ are compatible with respect to the
Rudin--Keisler preorder.

\begin{lemma}\label{RK_compatible}
Let $X,Y$ be spaces in \hs, and $(A^1_n)_{n\in\omega},
(A^2_n)_{n\in\omega}$ be two sequences of non-empty subsets of
$\omega$ such that, for each $k,n\in\omega, k\ne n$, i=1,2,
$A^i_k\cap A^i_n = \emptyset$. Put $$U = \bigcup_{n\in\omega}
\clo_{X\times Y}(A^1_n\times A^2_n)$$ and let
$f_1,f_2:\omega\to\omega$ be functions such  that, for $i=1,2$, $f_i(A^i_n) =
\{2n\}$, $f_i$ takes odd values on $\omega\setminus
\bigcup_{n\in\omega} A^i_n$, and is injective  on this complement. 

If $(u_1,u_2)\in
\clo_{X\times Y}(U) \setminus U$, then $u_1,u_2\in\omega^*$ and
$f_1(u_1)=f_2(u_2)\in\omega^*$, hence the ultrafilters $u_1$ and $u_2$ are compatible with respect to the Rudin--Keisler preorder. Moreover, $$u_1\in
\clo_{X}\big(\bigcup_{n\in\omega}A^1_n\big)\setminus
\bigcup_{n\in\omega}\clo_{X}(A^1_n) $$ and $$u_2\in
\clo_{Y}\big(\bigcup_{n\in\omega}A^2_n\big)\setminus
\bigcup_{n\in\omega}\clo_{Y}(A^2_n).$$
\end{lemma}
\begin{proof}
Let $(u_1,u_2)\in  \clo_{X\times Y}(U) \setminus U$.  For any
$n\in\omega$, the intersections $U\cap(\{n\}\times Y)$ and $U\cap(X
\times \{n\})$ are either empty or, for some $k\in\omega$, are equal
to $\{n\}\times \clo_Y (A^2_k)$ or $\clo_X (A^1_k)\times\{n\}$,
respectively. Therefore they are closed in $\{n\}\times Y$ and $X
\times \{n\}$, respectively. Since the sets $\{n\}\times Y$ and $X\times\{n\}$ are open in $X\times Y$, it follows that $u_1,u_2\in\omega^*$.

Let $V = \bigcup_{n\in\omega} A^1_n\times A^2_n$. Obviously,  $V$ is
dense in $U$, hence $(u_1,u_2)$ belongs to  $\clo_{X\times
Y}(V)\setminus U$. Set $p_i=f_i(u_i)$,  $i=1,2$.  If $p_1\ne p_2$,
then we can find $C\subset \omega$ such that $C\in p_1$ and
$(\omega\setminus C)\in p_2$. We claim that the neighborhood
$$O=\clo_{X}(f_1^{-1}(C))\times \clo_{Y}(f_2^{-1}(\omega\setminus
C))$$   of $(u_1,u_2)$ is disjoint with $V$.  Indeed, suppose that there are
$n\in\omega$ and
$$ (m_1,m_2)\in (A^1_n\times A^2_n)\cap \big( \clo_X(f_1^{-1}(C))\times\clo_Y(f_2^{-1}(\omega\setminus C))\big),$$
i.e., $m_1\in \clo_X(f_1^{-1}(C))$ and $m_2\in
\clo_Y(f_2^{-1}(\omega\setminus C))$. Since $\omega$ is a discrete subspace of
$X$ and $Y$, we conclude that $m_1\in f_1^{-1}(C)$ and $m_2\in
f_2^{-1}(\omega\setminus C)$, which means $2n=f_1(m_1)\in C$ and
$2n=f_2(m_2)\in\omega\setminus C$, a contradiction. Therefore $O\cap
V=\emptyset$, which together with $(u_1,u_2)\in O$ and
$(u_1,u_2)\in\clo_{X\times Y}(V)$ again leads to a contradiction.

Hence $p_1=p_2 = p$, and it remains to observe that $p\in\omega^*$.
Indeed, if $\{n\}\in p$, then we have two cases:

If $n$ is odd, then $|f_i^{-1}(n)|\le 1$, so  $f_i^{-1}(n)$ cannot belong to $ u_i\in\omega^*,\ i=1,2$.

If $n=2k$, then $A^i_k\in u_i,\ i=1,2$, hence  $$(u_1,u_2)\in
\clo_{X\times Y}(A^1_k\times A^2_k)\subset U,$$ which again gives a
contradiction.

 Finally, it is clear that  $u_1\in
\clo_{X}(\bigcup_{n\in\omega}A^1_n)$. If $u_1 \in \clo_{X}(A^1_n)$
for some $n$, we would get that $\{2n\}=f_1(A^1_n)\in f_1(u_1)=p$,
thus contradicting $p\in\omega^*$. Analogously for $u_2$.
 
\end{proof}

\begin{lemma}\label{jns_clopen}
Let $(\mu_n)_{n\in\omega}$ be a \bjns\ on a space $X$. If $Y$ is a clopen subset of $X$ such that the sequence $(|\mu_n|(Y))_{n\in\omega}$ does not converge to $0$, then there exist an increasing sequence $(n_k)$ and a \bjns\  $(\nu_k)_{k\in\omega}$ on $X$ such that $\supp\nu_k \subset Y\cap\supp \mu_{n_k}$ for every $k\in\omega$, in particular $(\nu_k)_{k\in\omega}$ is a \bjns\ on $Y$.
\end{lemma}
\begin{proof}
Pick an $a>0$ and an increasing sequence $(n_k)$  such that
$|\mu_{n_k}|(Y))> a$ for all $k\in\omega$. Let $$\nu_k =
(\mu_{n_k}\rstr Y)/|\mu_{n_k}|(Y).$$ Since any continuous bounded
function on $Y$ can be extended to a continuous bounded function on
$X$ by declaring value $0$ outside $Y$, it easily follows that
$(\nu_k)_{k\in\omega}$ is a \bjns\ on $X$ and on $Y$.
\end{proof}

\begin{lemma}\label{refining_jns}
Let $X$ and $Y$ be spaces in \hs\ such that all ultrafilters  in
$X\setminus\omega$ and $Y\setminus\omega$ are weak $P$-points. If
$C_{p}(X\times Y)$ has the  \BJNP\ (\JNP), then there exist a \bjns\
(a \jns) $(\mu_n)_{n\in\omega}$ on $X\times Y$, and sequences
$(A^1_n)_{n\in\omega}, (A^2_n)_{n\in\omega}$ of finite subsets of
$X,Y$, respectively, such that
\begin{enumerate}
\item $A^i_k\cap A^i_n = \emptyset$ for $k,n\in\omega,\ k\ne n$, $i=1,2$;

\medskip

\item $A^1_k\cap A^2_n = \emptyset$ for $k,n\in\omega, k\ne n$;

\medskip

\item $\supp \mu_n \subset A^1_n\times A^2_n$ for each $n\in\omega$;

\medskip

\item $\bigcup_{n\in\omega} A^1_n \cup \bigcup_{n\in\omega} A^2_n$
is a discrete subspace of $\beta\omega$.
\end{enumerate}
\end{lemma}
\begin{proof}
First we will consider the \BJNP\ case. Let  $(\nu_k)_{k\in\omega}$
be a \bjns\ on $X\times Y$. Given $n\in\omega$, the subspaces
$\{n\}\times Y$ and $X\times\{n\}$ are homeomorphic to $Y$ and $X$, 
respectively, so by Propositions \ref{HS_properties}.(1) and \ref{HS_properties}.(3) the spaces
$C_{p}(\{n\}\times Y)=C_{p}^*(\{n\}\times Y)$ and $C_{p}^*(X\times\{n\})=C_{p}(X\times\{n\})$ do not have the
\BJNP. Since these subspaces
are clopen, Lemma \ref{jns_clopen} implies that
$$\lim\limits_{k}|\nu_k|(\{n\}\times Y) = 0,\,\,\
\lim\limits_{k}|\nu_k|(X\times\{n\}) = 0.$$ Let $\pi_X: X\times Y\to
X$ and $\pi_Y: X\times Y\to Y$ be projections. For a subset $B$ of
$\omega $, we denote $\pi_X^{-1}(B)\cup \pi_Y^{-1}(B)$ by
$\mathcal{C}(B)$. Clearly, for any finite $B\subset\omega$, we have
$\lim\limits_{k}|\nu_k|(\mathcal{C}(B)) = 0$.

By induction we choose  an increasing sequence $(k_i)_{i\in\omega}$
and a sequence $(B_i)_{i\in\omega}$ of pairwise disjoint finite
subsets of $\omega$ such that
\begin{itemize}
\item[$(i)$]  $(\omega\times\omega)\cap\supp \nu_{k_i} \subset \mathcal{C}(\bigcup_{j=0}^{i-1} B_j)\cup (B_i\times B_i)$;

\medskip

\item[$(ii)$] $|\nu_{k_i}|(\mathcal{C}(\bigcup_{j=0}^{i-1} B_j)) < 1/(i+1)$.
\end{itemize}
For $i\in\omega$,  let $$Z_i =(X\times Y)\setminus
\mathcal{C}\big(\bigcup_{j=0}^{i-1} B_j\big).$$ Define $$\lambda_i =
(\nu_{k_i}\rstr Z_i)/\|\nu_{k_i}\rstr Z_i\|.$$ Using condition $(ii)$ above,
one can easily verify that $\lim\limits_{i}\|\lambda_i - \nu_{k_i}\|
= 0$, hence $(\lambda_i)_{i\in\omega}$ is a \bjns\ on $X\times Y$.
The first condition implies that $$(\omega\times\omega)\cap\supp
\lambda_i\subset B_i\times B_i.$$
Let $$C = \big(\pi_X\big(\bigcup_{i\in\omega} \supp \lambda_i\big)\cup \pi_Y\big(\bigcup_{i\in\omega} \supp \lambda_i\big)\big) \setminus \omega \subset \omega^*.$$ By Lemma \ref{lemma_Ma} we can find an infinite subset $T\subset \omega$ such that
$$C\cap\big(\clo_{\beta\omega}\bigcup_{i\in T}B_i\big)=\emptyset.$$  Then the
 sequence $(\lambda_i)_{i\in T}$ has the property that the set
 $$D:  = \pi_X\big(\bigcup_{i\in T} \supp \lambda_i\big)\: \cup\:  \pi_Y\big(\bigcup_{i\in T} \supp \lambda_i\big)$$
is a discrete subspace of $\beta\omega$. Indeed, the set $D'= D\cap\omega^*$ is a discrete subspace of $\beta\omega$ because it is countable and consists of weak
$P$-points (and of course the closure of $D'$ is contained in the closed set $\omega^*$, therefore it is disjoint with any subset of $\omega$, in particular $D\cap\omega$). In addition,  $\clo_{\beta\omega}(D\cap \omega)$ is
disjoint with  $D'$ by our choice of $T$ because
$D\cap\omega\subset\bigcup_{i\in T}B_i$. Thus $D$ is a union of two
discrete subsets of $\beta\omega$ so that the closure of any of them
does not intersect the other one, and hence it is also discrete.

\smallskip

Set $D_X=\pi_X(\bigcup_{i\in T} \supp \lambda_i)$ and
$D_Y=\pi_Y(\bigcup_{i\in T} \supp \lambda_i)$, so that $D=D_X\cup
D_Y$. 
 Fix a point $p\in D_X$, and find a clopen  set $V$ in $X$ such
that $V\cap D_X = \{p\}$. Let $U = V\times Y$. $U$ is clopen in
$X\times Y$ and $$U\cap \bigcup_{i\in T} \supp \lambda_i\subset
\{p\}\times Y.$$ Observe that each continuous function on
$\{p\}\times Y$ is bounded (due to the pseudocompactness of $Y$) and thus can be extended to a bounded continuous function on
$U$, therefore, by Lemma  \ref{jns_clopen}, we conclude that
$$\lim\limits_{i\in T}|\lambda_i|(\{p\}\times Y) = 0.$$ In the same
way we can show that, for any $q\in D_Y$, we have $\lim\limits_{i\in
T}|\lambda_i|(X\times \{q\}) = 0$.

Now, we can repeat our inductive construction from the first part of the proof. We choose an increasing sequence $(i_n)_{n\in\omega},\ i_n\in T$, and sequences $(A^1_n)_{n\in\omega}, (A^2_n)_{n\in\omega}$ of finite subsets of $D_X,D_Y$, respectively, such that the following conditions are satisfied (where $E_n = \bigcup_{l=1,2}\bigcup_{j=0}^{n-1} A^l_j$):
\begin{itemize}
\item[$(a)$] $(A^1_n\cup A^2_n)\cap E_n = \emptyset$;

\medskip

\item[$(b)$] $\supp \lambda_{i_n}\subset \big(\pi_X^{-1}(X\cap  E_n)\cup \pi_Y^{-1}(Y\cap E_n)\big)\cup (A^1_n\times A^2_n)$;

\medskip

\item[$(c)$] $|\lambda_{i_n}|\big(\pi_X^{-1}(X\cap  E_n)\cup \pi_Y^{-1}(Y\cap E_n)\big) < 1/(n+1)$.
\end{itemize}

For $n\in\omega$, let
$$S_n =(X\times Y)\setminus \big(\pi_X^{-1}(X\cap  E_n)\cup \pi_Y^{-1}(Y\cap E_n)\big).$$

\smallskip

Define $\mu_n = (\lambda_{i_n}\rstr S_n)/\|\lambda_{i_n}\rstr S_n\|$\,.
Conditions (1) and (2) %from the statement of the lemma
follow from $(a)$. Condition $(c)$ implies that
$$\lim\limits_{n}\|\lambda_{i_n} - \mu_{n}\| = 0,$$ hence
$(\mu_n)_{n\in\omega}$ is a \bjns\ on $X\times Y$. From condition
$(b)$ we deduce  that $$\supp \mu_n\subset A^1_n\times A^2_n.$$
Condition (4) %from the statement of the lemma
follows from
inclusions $A^1_n\subset D_X$, $A^2_n\subset D_Y$.

Observe that the union of supports of measures $\mu_{n}$ that we
constructed  above is contained in  the union of supports of
measures $\nu_k$ from our initial \bjns. If we start with a \jns\
$(\nu_k)_{k\in\omega}$, then, by Proposition
\ref{union_supp_bounded}, $\bigcup_{k\in\omega}\supp(\nu_k)$ is
bounded in $X\times Y$, therefore, if we repeat the above argument,
the resulting \bjns\ $(\mu_n)_{n\in\omega}$ will be also a \jns. Indeed, let $f\in C(X\times Y)$ and $M>0$ be such that $|f(x)|<M$ for every $x\in\bigcup_{k\in\omega}\supp(\nu_k)$. Fix a retraction $g:\mathbb{R}\to[-M,M]$. Then, $g\circ f\in C_p^*(X\times Y)$ and $\mu_k(f)=\mu_k(g\circ f)$ for each $k\in\omega$, so $\lim\limits_{n}\mu_k(f)=0$.
\end{proof}

Before we prove Theorem~\ref{main_ex}, which is the main result of
this section, we will show its weaker version, since it has an
essentially simpler proof and a less technical set-theoretic
assumption.

\begin{theorem}
Assume that there exist two incompatible with respect to the Rudin--Keisler preorder weak $P$-points $u,v$ in $\omega^*$. Then there exist infinite pseudocompact spaces $X,Y$  in \hs\ such that $C_{p}(X\times Y)$ does not have the \BJNP, and hence it does not have the \JNP.
\end{theorem}
\begin{proof} It is an easy and well known observation that, given an ultrafilter $w\in \omega^*$ and a set $A\in [\omega]^\omega$, there is an ultrafilter $w'$ in the closure of $A$ in $\beta\omega$, which is isomorphic to $w$. Therefore, we may take spaces $$X = \omega\cup\{u_A: A\in [\omega]^\omega\},\,\,\ Y = \omega\cup\{v_A: A\in [\omega]^\omega\}$$ in \hs\ such that all ultrafilters $u_A$ are isomorphic to $u$, and all ultrafilters $v_A$ are isomorphic to $v$. We will show that the product $X\times Y$ does not have a \bjns.

Suppose  the contrary, and let $$(\mu_n)_{n\in\omega},\,\,\,
(A^1_n)_{n\in\omega},\,\,\, (A^2_n)_{n\in\omega}$$ be sequences of
measures and sets given by Lemma \ref{refining_jns}. Let $$A^i =
\bigcup_{n\in\omega} A^i_n, \,\, i=1,2.$$ Recall that basic clopen
sets in $X$ and $Y$ are closures of subsets of $\omega$. For
$i=1,2$, since $A^i$ is discrete, using a simple induction, we can
find  a family $\{U^i_p: p\in A^i\}$ of pairwise disjoint
subsets of $\omega$, such that the closure of $U^i_p$ is a
neighborhood of $p$. For $i=1,2$ and $n\in\omega$, let $$V^i_n =
\bigcup\{U^i_p: p\in A^i_n\}.$$ For fixed $i$, the family
$\{V^i_n:  n\in\omega\}$ is disjoint, and we have
$$\supp\mu_n\subset \clo_{X\times Y}(V^1_n\times V^2_n).$$ From our
construction of spaces $X$ and $Y$ and Lemma \ref{RK_compatible} we
infer that the set $$U = \bigcup_{n\in\omega} \clo_{X\times
Y}(V^1_n\times V^2_n)$$ is clopen in $X\times Y$. Let $f_n:
\supp\mu_n\to [-1,1]$ be defined by  $$f_n(z) =
\mu_n(\{z\})/|\mu_n(\{z\})|$$ for $z\in\supp\mu_n$. Since
$\supp\mu_n$ is finite, we can extend each $f_n$ to a continuous
function 
$$F_n: \clo_{X\times Y}(V^1_n\times V^2_n)\to [-1,1].$$ Since the sets $\clo_{X\times Y}(V^1_n\times V^2_n)$, $n\in\omega$, are clopen and pairwise disjoint, the union of all $F_n$ is continuous on $U$, and we can extend this union to a bounded continuous $F$ on $X\times Y$, declaring value $0$ outside $U$. Clearly, for all $n\in\omega$, we have $\mu_n(F) = 1$, a contradiction.
\end{proof}

For a set $A$ by $\bigtriangleup_A$ we denote the diagonal $\{(a,a): a\in A\}$ in $A\times A$.

\begin{lemma}\label{avoid_diagonal}
Let $E,F$ be non-empty finite subsets of a set $X$, and $\mu$ be a signed measure on $E\times F$. Then there exist disjoint $G\subset E$ and $H\subset F$ such that $$|\mu|(G\times H)\ge |\mu|((E\times F)\setminus \bigtriangleup_{X})/6.$$
\end{lemma}
\begin{proof}
Let $A= E\cap F$. 
First, we will show that there exist disjoint subsets $B,C$ of $A$ such that $$|\mu|(B\times C)\ge |\mu|((A\times A)\setminus \bigtriangleup_{A})/4.$$
The case when $|A|\le 1$ is easy, so we also assume that $n=|A|>1$.
Suppose the contrary, then, for any partition $(B,C)$ of $A$, we have $$|\mu|(B\times C) < |\mu|((A\times A)\setminus \bigtriangleup_{A})/4.$$

Let $\mathcal{P}$ be the family of all partitions $(B,C)$ of $A$ (we allow trivial partitions with one empty set). We have $|\mathcal{P}| = 2^n$ and each point from $(A\times A)\setminus \bigtriangleup_A$ appears in $2^{n-2}$ sets $B\times C$, where $(B,C)\in \mathcal{P}$. Hence we have
$$
2^{n-2}|\mu|((A\times A)\setminus \bigtriangleup_A) = \sum\{|\mu|(B\times C): (B,C)\in \mathcal{P}\}
<  2^n(|\mu|((A\times A)\setminus \bigtriangleup_{A})/4)
$$
which gives a contradiction.
\medskip

Next, observe that $$(E\times F)\setminus \bigtriangleup_{X} = \big((E\setminus F)\times F\big)\cup\big(A\times (F\setminus E)\big)\cup\big((A\times A)\setminus \bigtriangleup_A\big),$$ and the sets $A, E\setminus F, F\setminus E$ are all disjoint. If the variation $|\mu|$ of one of the rectangles $(E\setminus F)\times F$ and $A\times(F\setminus E)$ is at least equal to $|\mu|((E\times F)\setminus\bigtriangleup_{X})/6$, then we can take this rectangle as $G\times H$. If this is not the case, then we have $$|\mu|((A\times A)\setminus\bigtriangleup_A)>(4/6)\cdot|\mu|((E\times F)\setminus\bigtriangleup_{X}),$$ which implies that $$|\mu|(B\times C)>(1/4)\cdot(4/6)\cdot|\mu|((E\times F)\setminus\bigtriangleup_{X}),$$ and we can take $G=B$ and $H=C$.
\end{proof}

\begin{lemma}\label{avoid_null_seq}
Let $(\mu_n)_{n\in\omega}$ be a \bjns\ on a space $X$. If $Z$ is a subset of $X$ such that  $\lim\limits_{n}\|\mu_n\rstr Z\| = 0$, then there exist an increasing sequence $(n_k)$ and a \bjns\  $(\nu_k)_{k\in\omega}$ on $X$ such that $\supp\nu_k \subset \supp \mu_{n_k}\setminus Z$ for every $k\in\omega$.
\end{lemma}
\begin{proof}
Take an increasing sequence $(n_k)$ that $\|\mu_{n_k}\rstr Z\|< 1$ for all $k$. Let $Y= X\setminus Z$. Define $\nu_k = (\mu_{n_k}\rstr Y)/\|\mu_{n_k}\rstr Y\|$. We have $$\lim\limits_{k}\|\nu_k - \mu_{n_k}\| = 0,$$ and  hence one has that $(\nu_k)_{k\in\omega}$ is a \bjns\ on $X$.
\end{proof}

Finally, we present the main result of this section. It is proved
under the following set-theoretic assumption
\begin{itemize}
\item[($\dagger$)] {\em
 There exists a function $A\mapsto u_A$ assigning to each $A\in
[\omega]^\omega$ a weak $P$-point $u_A\ni A$ on $\omega$ such that
for every pair $\langle f_1,f_2\rangle$ of functions from $\omega$
to $\omega$, there exists $\mathcal A\subset [\omega]^\omega$ of
size $|\mathcal A|<\mathfrak c$ such that for all
$A_1\not\in\mathcal A$ and $A_2\in [\omega]^\omega\setminus\{A_1\}$
we have $f_1(u_{A_1})\neq f_2(u_{A_2})$, provided that
$f_1(u_{A_1}), f_2(u_{A_2})\in\omega^*$.}
\end{itemize}

\medskip

We do not know whether $(\dagger)$ can be proved outright in ZFC. However, at
the end of this section we present a constellation of cardinal
characteristics of the continuum which implies $(\dagger)$ and holds in many standard models of set theory as well as is implied by the Continuum Hypothesis or Martin's axiom.% and holds in all models of $\mathfrak d=\mathfrak c\leq \omega_2$.

\begin{theorem}\label{main_ex}
Let $[\omega]^\omega\ni A\mapsto
u_A\in\clo_{\beta\omega}(A)\cap\omega^*$ be a witness for $(\dagger)$
and $X=\omega\cup\{u_A:A\in [\omega]^\omega\}$. Then  $C_p(X\times
X)$ does not have the \BJNP, and hence it does not have the \JNP.
\end{theorem}
\begin{proof}
Suppose the contrary, and let $(\mu_n)_{n\in\omega}$,
$(A^1_n)_{n\in\omega}$, and $(A^2_n)_{n\in\omega}$ be sequences of
measures and sets given by Lemma \ref{refining_jns}. Let $$A=
\bigcup_{i=1,2}\bigcup_{n\in\omega} A^i_n.$$  By condition (4) of
Lemma \ref{refining_jns}, the set  $A$ is
discrete.

First we will show that
\[\tag{$*$}\lim\limits_{n}|\mu_n|((A^1_n\times A^2_n)\setminus \bigtriangleup_X) = 0.\]

Assume towards the contradiction that there exist an $a>0$ and an increasing sequence $(n_k)$ such that  $$|\mu_{n_k}|((A^1_{n_k}\times A^2_{n_k})\setminus \bigtriangleup_X))> a$$ for all $k\in\omega$. Using Lemma \ref{avoid_diagonal}, for each  $k\in\omega$, we can find disjoint sets $B^i_k \subset A^i_{n_k},\ i=1,2$, such that $$|\mu_{n_k}|(B^1_k\times B^2_k) > a/6.$$
Let $B^i = \bigcup_{k\in\omega} B^i_k$ for  $i=1,2$. From conditions (1), and (2) of Lemma \ref{refining_jns}, it follows that the sets $B^1,B^2\subset A$ are disjoint.

Since the set $A$ is discrete, we can find a disjoint family $\{U_p: p\in A\}$ of subsets of $\omega$, such that the closure of $U_p$ in $\beta\omega$ is a neighborhood of $p$.

For $i=1,2$ and $k\in\omega$, let $$V^i_k = \bigcup\{U_p: p\in B^i_k\} \text{ and }  W^i_k = \bigcup\{U_p: p\in A^i_{n_k}\}.$$  These sets have the following properties:
\medskip

\begin{enumerate}
\item For fixed $i$, the family  $\{W^i_k:  k\in\omega\}$ is disjoint;

\medskip

\item $B^1_k\times B^2_k\subset \clo_{X\times X}(V^1_k\times V^2_k)\subset \clo_{X\times X}(W^1_k\times W^2_k)$;

\medskip

\item $\supp \mu_{n_k}\subset \clo_{X\times X}(W^1_k\times W^2_k)$.
\end{enumerate}
Moreover, the sets $V^i = \bigcup_{k\in\omega} V^i_k,\ i=1,2$, are
disjoint,  hence have disjoint closures in $X$. Consider maps
$f_i:\omega\to\omega$ such that $f_i^{-1}(\{2k\})=V^i_k$ and
$f_i\upharpoonright(\omega\setminus V^i)$ is
injective and takes odd values, where $i=1,2$. Set $$U' = \bigcup_{k\in\omega}
\clo_{X\times X}(V^1_k\times V^2_k)$$ and note that 
  if $(u_1,u_2)\in\clo_{X\times X}(U') \setminus U'$, then $u_1\ne u_2$ and, by
  Lemma~\ref{RK_compatible},
$f_1(u_1)=f_2(u_2)\in\omega^*$. Applying $(\dagger)$ to $(f_1,f_2)$
we get
 $\mathcal A\subset [\omega]^\omega$ of size $|\mathcal
A|<\mathfrak c$ such that for all $A_1\not\in\mathcal A$ and $A_2\in
[\omega]^\omega\setminus\{A_1\}$ we have $f_1(u_{A_1})\neq
f_2(u_{A_2})$, provided that $f_1(u_{A_1}),
f_2(u_{A_2})\in\omega^*$. It follows from the above that
$$ \clo_{X\times X}(U') \setminus U' \subset \{u_A:A\in\mathcal A\}^2. $$
Applying Lemma~\ref{lemma_Ma} we can find $I\in [\omega]^\omega$
such that the set $$\clo_X\big(\bigcup_{k\in I}V^i_k\big)\setminus \bigcup_{k\in
I}\clo_X (V^i_k)$$ is disjoint from $\{u_A:A\in\mathcal A\}$, where
$i=1,2$. Thus, the set $$U = \bigcup_{k\in I} \clo_{X\times X}(V^1_k\times
V^2_k)$$ is clopen in $X\times X$.   Indeed, since $U$ is clearly open, it remains to prove
that it is also closed. Otherwise there exists
$(u_1,u_2)\in\clo_{X\times X}(U)\setminus U$. Since for every
$n\not\in I$ we have that  $\clo_{X\times X}(V^1_n\times V^2_n)$ is
a clopen subset of $X\times X$ disjoint from $U$, it is also
disjoint from $\clo_{X\times X}(U)$, and hence $$\clo_{X\times
X}(U)\setminus U\subset \clo_{X\times X}(U')\setminus U',$$ which
yields $$ (u_1,u_2) \in \{u_A:A\in\mathcal A\}^2.
$$
In particular $u_1\in \{u_A:A\in\mathcal A\}$. However,
\begin{eqnarray*}
u_1 \in  \clo_{X}\big(\bigcup_{k\in I} V^1_k\big)\setminus \bigcup_{k\in I}
\clo_{X}(V^1_k)
\end{eqnarray*}
by the last clause of Lemma~\ref{RK_compatible}, which is impossible
by our choice of $I$. 

For every $k\in I $ set $C_k = \supp \mu_{n_k}\cap U$. Properties
(1)--(3) imply the inclusions
$$\supp \mu_{n_k}\cap (B^1_k\times B^2_k)\subset C_k\subset \clo_{X\times X}(V^1_k\times V^2_k)\,.$$
Let $f_k: C_k\to [-1,1]$ be defined  by  $$f_k(z) =
\mu_{n_k}(\{z\})/|\mu_{n_K}(\{z\})|$$ for $z\in C_k$. Since $C_k$ is
finite, we can extend each $f_k$ to a continuous function $$F_k:
\clo_{X\times X}(V^1_k\times V^2_k)\to [-1,1].$$ Since the sets
$\clo_{X\times Y}(V^1_k\times V^2_k)$, $k\in\omega$, are clopen and pairwise disjoint, the union of all
$F_k$, $k\in I$, is continuous on $U$, and we can extend this union
to a bounded continuous $F$ on $X\times X$, declaring value $0$
outside $U$. Clearly, for all $k\in I$, we have $\mu_{n_k}(F) >
a/6$, a contradiction.

Now, condition ($*$) and Lemma \ref{avoid_null_seq} imply  that
$X\times X$ admits a \bjns\  $(\nu_k)_{k\in\omega}$ supported on
$\bigtriangleup_{X}$. Clearly, $\bigtriangleup_{X}$ is homeomorphic
to $X$, and each continuous bounded function on $\bigtriangleup_{X}$
extends to a continuous bounded function on $X\times X$. We obtain a
contradiction with the lack of \BJNP\ for $C_p(X)$, see Propositions \ref{bjnp_beta_c_0} and
\ref{HS_properties}.(3).
\end{proof}

Now, we present a sufficient condition for $(\dagger)$. Recall
that $\mathfrak u$ denotes the minimal cardinality of a base of neighborhoods for an
ultrafilter in $\omega^*$, and $\mathfrak d$ is the minimal
cardinality of a cover of $\omega^\omega$ by compact sets.
 These cardinal characteristics can consistently attain
arbitrary uncountable regular values $\le\mathfrak{c}$, independently one from another,
 see \cite{BlaShe89}. If $\mathfrak u<\mathfrak d$, then $\mathfrak
 d$ is regular, see \cite{Vau90} and \cite{Bla10} for more information on these as well
as other cardinal characteristics of the continuum.

Recall that $u\in\omega^*$ is a \emph{$P$-point}, if for
any countable family $\mathcal B$ of open neighbourhoods of $u$ in
$\omega^*$ the intersection $\cap\mathcal B$ contains $u$ in its
interior. Clearly, every $P$-point is also a weak $P$-point.
$P$-points exist under various set-theoretic assumptions, e.g.,
under $\mathfrak d=\mathfrak c$ every filter on $\omega$ generated
by $<\mathfrak c$ many sets can be enlarged to a $P$-point, see \cite{Ket76}. On the
other hand, there are models  of ZFC without $P$-points (see \cite{Wim}).
In contrast with $P$-points the weak $P$-points exist in ZFC by
\cite{weak_P_exist}.

A family $\mathcal A\subset[\omega]^\omega$ is called
\emph{strongly centered}, if $\cap\mathcal A'\in [\omega]^\omega$ for any
finite $\mathcal A'\subset\mathcal A$. A sequence of elements of
$[\omega]^\omega$ is strongly  centered if the set of its elements is
strongly centered.

\begin{lemma}\label{suff_for_dag}
If $\mathfrak d=\mathfrak c\leq \mathfrak u^+$, then $(\dagger)$
holds. Moreover, there is a map  $A\mapsto u_A$  witnessing 
$(\dagger)$ such that $u_A$ is a $P$-point for every $A\in
[\omega]^\omega$.
 \end{lemma}
\begin{proof}
We need the following auxiliary
\begin{sublemma} \label{cl01}
Suppose that $\kappa$ is a cardinal such that $\kappa \leq \mathfrak
u$ and $\kappa<\mathfrak c=\mathfrak d$. Suppose that
$\{u_\xi:\xi<\kappa\}\subset\omega^*$ and $\{\langle
f^\xi_1,f^\xi_2\rangle:\xi<\kappa\}$ is a family of pairs of  maps from
$\omega$ to $\omega$. Then for every $A\in [\omega]^\omega$ there
exists a $P$-point $u\ni A$ such that $ f^\xi_1(u)\neq f^\xi_2(u_\xi)$ for
all $\xi<\kappa$, provided that $f^\xi_1(u),f^\xi_2(u_\xi)\in \omega^*$.
\end{sublemma}
\begin{proof}
Without loss of generality we may assume  that
$f^\xi_2(u_\xi)\in\omega^*$ for all $\xi<\kappa$. Fix
$A\in[\omega]^\omega$. Set $A_{-1}=A$ and suppose that for some
$\eta<\kappa$ we have already constructed a strongly  centered
 sequence $\mathcal A_\eta=\langle
A_\xi:-1\leq \xi<\eta\rangle$  of infinite subsets of $\omega$  with the following
property (we associate $\xi$ with the set $\{\alpha:\alpha<\xi\}$):
\begin{itemize}
\item[$(\ast)_\eta$] For every $\xi<\eta$, if $f^\xi_1\big(\bigcap_{\gamma\in a}A_\gamma\big)\in
[\omega]^\omega$ for all $a\in[\xi\cup\{-1\}]^{<\omega}$, then
$A_\xi=(f^\xi_1)^{-1}(B_\xi)$ for some $B_\xi\in [\omega]^\omega\setminus
f^\xi_2(u_\xi)$ (hence  $ f^\xi_1(A_\xi) \not\in f^\xi_2(u_\xi)$).
\end{itemize}
If $f^\eta_1\big(\bigcap_{\gamma\in a}A_\gamma\big)$ is finite for some
$a\in[\eta\cup\{-1\}]^{<\omega}$, then we simply set $A_\eta=\omega$. If
$f^\eta_1\big(\bigcap_{\gamma\in a}A_\gamma)$ is infinite for all
$a\in[\eta\cup\{-1\}]^{<\omega}$ and $f^\eta_1\big(\bigcap_{\gamma\in
a_\eta}A_\gamma)\not\in f^\eta_2(u_\eta)$  for some finite
$a_\eta\in[\eta\cup\{-1\}]^{<\omega}$, then we set
$$A_\eta=(f^\eta_1)^{-1}(f^\eta_1\big(\bigcap_{\gamma\in a_\eta}A_\gamma\big)).$$ Otherwise, notice
that for the family
$$\mathcal Y=\Big\{ f^\eta_1\big(\bigcap_{\gamma\in a}A_\gamma\big)\: :\: a\in[\eta\cup\{-1\}]^{<\omega}\Big\}\subset
f^\eta_2(u_\eta)$$ 
the family $$\{\clo_{\beta\omega}(Y)\setminus Y: Y\in \mathcal{Y}\}$$ cannot be a base of neighborhoods of $f^\eta_2(u_\eta)$ in $\omega^*$, because it has size
$<\mathfrak u$, and thus there exists $X\in f^\eta_2(u_\eta)$ with
 $|Y\setminus X|=\omega$ for all $Y\in\mathcal Y$. Set
$B_\eta=\omega\setminus X$ and $A_\eta=(f^\eta_1)^{-1}(B_\eta)$. It follows from
the above that $\mathcal A_{\eta+1}=\langle A_\xi:-1\leq \xi\leq \eta\rangle$
is strongly  centered and satisfies $(\ast)_{\eta+1}$. This completes our
recursive construction of a sequence $\mathcal A_{\kappa}=\langle
A_\xi:-1\leq \xi<\kappa \rangle$. Since $\kappa<\mathfrak c=\mathfrak
d$, there exists a $P$-point $u\supset\{A_\xi:-1\leq \xi<\kappa\}$, see
\cite[1.2 Theorem]{Ket76}. For every $\xi\in\kappa$ such that
$f^\xi_1(u)\in\omega^*$ we have that $f^\xi_1(A_\xi)\in f^\xi_1(u)\setminus
f^\xi_2(u_\xi)$, which completes our proof.
\end{proof}
Now we proceed with the proof of Lemma~\ref{suff_for_dag}. Let
 $\{A_\alpha:\alpha<\mathfrak c\}$ and $\{\langle f^\alpha_1, f^\alpha_2\rangle:\alpha<\mathfrak c\}$  be
enumerations of $[\omega]^\omega$ and the family of all pairs of
functions from $\omega$ to $\omega$, respectively.  Recursively over
$\alpha\in\mathfrak c$ we shall construct a $P$-point $u_\alpha\ni
A_\alpha$ as follows: Assuming that $u_\beta$ is already constructed
for all $\beta\in\alpha$, let us fix an enumeration $$\big\{\langle
u_\xi,(f^\xi_1,f^\xi_2) \rangle:\xi<\kappa\big\}$$ of $$\big\{\langle
u_\beta,(f^\gamma_{1},f^\gamma_{2})\rangle:\beta,
\gamma\in\alpha\big\},$$ where $\kappa=|\alpha\times\alpha|$. By Sublemma~\ref{cl01} there
exists a $P$-point $u_\alpha\ni A_\alpha$ such that
$f^\xi_1(u_\alpha)\neq f^\xi_2(u_\xi)$ for all $\xi<\kappa$,
provided that $f^\xi_1(u_\alpha),f^\xi_2(u_\xi),\in \omega^*$. In
other words,
$$\forall \beta,\gamma\in\alpha \:\big(f^\gamma_1(u_{\alpha}),f^\gamma_{2}(u_\beta)\in \omega^*
\Rightarrow f^\gamma_1(u_{\alpha})\neq f^\gamma_{2}(u_\beta)\big).$$
 The
map $A_\alpha\mapsto u_\alpha$ constructed this way clearly
witnesses $(\dagger)$.
 \end{proof}

\begin{remark} If there exists an $RK$-antichain of
size $\mathfrak c$ consisting of weak $P$-points, then $(\dagger)$
holds. Indeed, let $\mathcal{A}$ be such an antichain. Using a standard transfinite inductive construction, we can easily obtain a family $\{u_A: A\in [\omega]^\omega\}\subset \omega^*$ such that, for any $A\in [\omega]^\omega$, $A\in u_A$, and the ultrafilters $u_A$ and $u_B$ are isomorphic to distinct elements of $\mathcal{A}$ for  $A\ne B$. Clearly, the assignment $A\mapsto u_A$ is as required in $(\dagger)$. However, by Lemma~\ref{suff_for_dag},
$(\dagger)$ holds in all models of $\mathfrak c=\mathfrak
d\leq\omega_2$, in particular it holds in the Blass--Shelah model
constructed in \cite{BS} in which the $RK$-preorder is downwards
directed. Thus the existence of such $RK$-antichans as mentioned
above is sufficient but not necessary for $(\dagger)$ and thus also
for Theorem~\ref{main_ex}.
\end{remark}

Theorem \ref{lastest} follows now easily.

\begin{proof}[{\normalfont\textbf{Proof of Theorem \ref{lastest}}}]
The  theorem follows from Theorems \ref{main_ex} and \ref{bjnp_c_0}, Corollary \ref{corollary_cpstar_no_c0_cp_no_c0}, and the consistency of $(\dagger)$.
%see also the comments in the paragraph preceding Proposition \ref{not-pseudocompact}.
\end{proof}

\end{document}